\documentclass[a4paper,11pt]{article}
\usepackage[left=2cm,right=2cm,top=2cm,bottom=2cm]{geometry}
\usepackage{cite}
\usepackage{tikz}
\usetikzlibrary{arrows}
\usepackage[centertags]{amsmath}
\usepackage{amsfonts}
\usepackage{amssymb}
\usepackage{amsthm}
 \newtheorem{theorem}{Theorem}[section]
 \newtheorem{corollary}[theorem]{Corollary}
 \newtheorem{lemma}[theorem]{Lemma}
 \newtheorem{proposition}[theorem]{Proposition}
 
 \theoremstyle{definition}
 \newtheorem{definition}[theorem]{Definition}
 \theoremstyle{remark}
 \newtheorem{remark}[theorem]{Remark}
 \newtheorem{example}[theorem]{Example}
 \numberwithin{equation}{section}
\begin{document}
\begin{center}\Large{\textbf{On groupoid graded von Neumann regular rings and 
a Brandt groupoid graded Leavitt path algebras}}
\let\thefootnote\relax
\end{center}
\begin{center}\textbf{Emil Ili\'{c}-Georgijevi\'{c}}\end{center}
\begin{abstract}
Let $S$ be a partial groupoid, that is, a set with a partial binary operation. An 
$S$-graded ring $R$ is said to be graded von Neumann regular if $x\in xRx$ for 
every homogeneous element $x\in R.$ Under the assumption that $S$ is 
cancellative, we characterize $S$-graded rings which are graded von Neumann 
regular. If a ring is $S$-graded von Neumann regular, and if $S$ is cancellative, then 
$S$ is such that for every $s\in S,$ there exist $s^{-1}\in S$ and idempotent 
elements $e,$ $f\in S$ for which $es=sf=s,$ $fs^{-1}=s^{-1}e=s^{-1},$ 
$ss^{-1}=e$ and $s^{-1}s=f,$ which is a property enjoyed by Brandt groupoids. We 
observe a Leavitt path algebra of an arbitrary non-null directed graph over a unital 
ring as a ring graded by a Brandt groupoid over the additive group of integers 
$\mathbb{Z},$ and we prove that it is graded von Neumann regular if and only if its 
coefficient ring is von Neumann regular, thus generalizing the recently obtained result 
for the canonical $\mathbb{Z}$-grading of Leavitt path algebras.
\end{abstract}
\section{Introduction}\label{introduction}
Throughout the paper, all rings are assumed to be associative, and, unless otherwise 
stated, without a unity (non-unital). A ring with unity is said to be unital.\\
\indent Assigning various algebraic structures to directed graphs is widely present in
the literature (see \cite{avk,aasm}). Given a field $K$ and a directed graph $E,$ a 
specific $K$-algebra associated to $E$ can be constructed, called the 
\emph{Leavitt path algebra}, denoted by $L_K(E).$ Leavitt path algebras are 
introduced independently in \cite{aap,amp} as algebraic analogues of graph 
$C^*$-algebras. One obtains \emph{Leavitt algebras of type} $(1,n)$ \cite{wgl} as a 
particular case, see for instance \cite{aasm}. Leavitt path algebras over coefficient 
rings other than fields have been considered as well. Commutative unital rings are 
considered in \cite{mt}, the ring of integers in \cite{js}, and Leavitt path algebras 
over arbitrary unital rings in \cite{rh1}.\\
\indent Although introduced relatively recently, 
Leavitt path algebras have received a lot of attention, see for instance 
\cite{ga,aasm} and references therein. In particular, it is of interest to relate various 
combinatorial properties of the graph $E$ with the algebraic properties of the Leavitt 
path algebra of $E$ (as for instance in \cite{hv}). It is also of interest to investigate 
various algebraic properties of the Leavitt path algebra $L_R(E)$ with respect to the 
correspondig properties of the coefficient ring $R$ (as for instance in \cite{rh}).\\ 
\indent Recall that a ring $R$ is said to be \emph{von Neumann regular} if 
$x\in xRx$ for every $x\in R.$ It is proved in \cite{ar} that $L_K(E)$ is 
von Neumann regular if and only if $E$ is acyclic. If $R$ is a $G$-graded ring,
where $G$ is a group with identity $e,$ we recall that $R$ is said to be
\emph{graded von Neumann regular} \cite{noy2} if every homogeneous element 
$x\in R$ belongs to $xRx.$ (Note that the graded von Neumann regularity can be 
defined the same way for any grading set.) As it is well-known, 
since $G$ is a group, this is equivalent to the existence of 
a homogeneous element $y\in R$ such that $x=xyx.$ Now, for any ring $R$ and a
directed graph $E,$ the Leavitt path algebra $L_R(E)$ is equipped with a natural 
$\mathbb{Z}$-grading, induced by the lengths of the paths of 
$E,$ where $\mathbb{Z}$ denotes the additive group of integers, see for instance 
\cite{dl}. If $K$ is a field, it was reasonable to ask 
what could be said of the $\mathbb{Z}$-graded von Neumann regularity of 
$L_K(E),$ that is, whether it is true that for every homogeneous element 
$x\in L_K(E)$ there exists a homogeneous element $y\in L_K(E)$ such that 
$x=xyx.$ This question was raised and answered in the affirmative in \cite{rh}.
\begin{theorem}[Theorem~10 in \cite{rh}]\label{theoremrh}
For a field $K$ and a directed graph $E,$ the $\mathbb{Z}$-graded Leavitt path
algebra $L_K(E)$ is graded von Neumann regular.
\end{theorem}
Let us mention that, recently, in \cite{lv1,lv2}, a $\mathbb{Z}$-graded unit regular 
Leavitt path algebra $L_K(E)$ is characterized in terms of the properties 
of $E,$ as well as the other graded cancellation properties of $L_K(E),$ 
including the \emph{graded cleanness property} from \cite{igs} (see also 
\cite{eig11,eig7}).\\
\indent Let us return to graded von Neumann regularity. If $G$ is a group with 
identity $e$ and $R$ is a strongly $G$-graded ring with unity, then, it is well-known 
by Corollary~C.I.1.5.3 in \cite{noy2} that $R$ is graded von Neumann regular if and 
only if $R_e$ is von Neumann regular. This result was established by using Dade's 
theorem, and later on, Theorem~3 in \cite{hy} provided us with a direct, 
element-wise proof of this fact. Non-unital graded von Neumann regular rings are 
studied in \cite{rh} in order to get Theorem~\ref{theoremrh}, which is 
recently extended in \cite{dl} to the Leavitt path algebras over arbitrary 
unital rings with the help of a generalization of Theorem~3 in \cite{hy}. Namely, the 
following characterization holds.
\begin{theorem}[Theorem~1.4 in \cite{dl}]\label{theoremdl}
Let $R$ be a unital ring and $E$ a directed graph, which is not null. Then, the 
$\mathbb{Z}$-graded Leavitt path algebra $L_R(E)$ is graded von 
Neumann regular if and only if $R$ is von Neumann regular.
\end{theorem}
Now, let $K$ be a field and let $A_n$ be the oriented $n$-line graph with
$n$ vertices and $n-1$ edges:
\begin{center}\begin{tikzpicture}[node distance={20mm}]
\node(v_1){$v_1$};
\node(v_2)[right of=v_1]{$v_2$};
\node(v_3)[right of=v_2]{$v_3$};
\node(dots)[right of=v_3]{$\dots$};
\node(v_{n-1})[right of=dots]{$v_{n-1}$};
\node(v_n)[right of=v_{n-1}]{$v_n.$};
\draw[->] (v_1) -- node[midway, above]{$\alpha_1$} (v_2);
\draw[->] (v_2) -- node[midway, above]{$\alpha_2$} (v_3);
\draw[->] (v_3) -- node[midway, above]{$\alpha_3$} (dots);
\draw[->] (dots) -- node[midway, above]{$\alpha_{n-2}$} (v_{n-1});
\draw[->] (v_{n-1}) -- node[midway, above]{$\alpha_{n-1}$} (v_n);
\end{tikzpicture}\end{center}
It is well-known, and easy to check, that $L_K(A_n)$ is isomorphic to the full matrix 
ring $\mathbb{M}_n(K)$ (see Proposition~1.3.5 in \cite{aasm}). The
isomorphism is given by mapping each vertex $v_i$ to the standard matrix unit
$e_{i,i},$ and by mapping the edges $\alpha_i$ and the corresponding ghost edges 
$\alpha_i^*$ to $e_{i,i+1}$ and $e_{i+1,i},$ respectively.\\
\indent It is also well-known that $\mathbb{M}_n(K)$ can be observed as a 
$\mathbb{Z}$-graded ring $\bigoplus_{k\in\mathbb{Z}}(\mathbb{M}_n(K))_k,$ 
where $(\mathbb{M}_n(K))_k$ consists of matrices $(a_{i,j})$ for which 
$a_{i,j}=0$ whenever $i,j$ are such that $i-j\neq k,$ for $-(n-1)\leq k\leq n-1,$ and 
$(\mathbb{M}_n(K))_k$ is the zero matrix for $|k|\geq n.$ Moreover, the 
above isomorphism gives rise to a $\mathbb{Z}$-graded isomorphism between 
$L_K(A_n)$ and $\mathbb{M}_n(K).$ However, $\mathbb{M}_n(K)$ can be 
observed as a ring graded by a partial groupoid. In this article, by a groupoid we
mean a set with a binary operation.\\
\indent Let $(\mathbb{M}_n(K))_{(i,j-i,j)}$ be a subset of $\mathbb{M}_n(K)$ 
which consists of matrices with an entry from $K$ in the $(i,j)$ position and zeroes
elsewhere, where $i,j=1,\dots,n.$ Then, for the additive group of 
$\mathbb{M}_n(K)$ we have that 
$\mathbb{M}_n(K)=\bigoplus_{i,j=1}^n(\mathbb{M}_n(K))_{(i,j-i,j)},$ and,
moreover, $(\mathbb{M}_n(K))_{(i,j-i,j)}(\mathbb{M}_n(K))_{(k,l-k,l)}=\delta_{j,k}
(\mathbb{M}_n(K))_{(i,l-i,l)},$ where $\delta_{j,k}$ is the Kronecker delta. This 
induces a partial operation on the set $S=\{(i,j-i,j)\ |\ i,j=1,\dots,n\}$ given by
$(i,j-i,j)\cdot(k,l-k,l)=(i,l-i,l)$ if $j=k$ (see \cite{ak}). In fact, $S$ is a Brandt 
groupoid. For a vertex $v_i$ of $A_n,$ let us put $w(v_i)=(i,0,i)\in S,$ for an edge 
$\alpha_i,$ put $w(\alpha_i)=(i,1,i+1),$ and for the corresponding ghost edge 
$\alpha_i^*,$ put $w(\alpha_i^*)=(i+1,-1,i).$ Then, it is easy to see that 
$L_K(A_n)$ is also an $S$-graded ring 
$\bigoplus_{(i,j-i,j)\in S}(L_K(A_n))_{(i,j-i,j)},$ where $(L_K(A_n))_{(i,j-i,j)}$ is the 
$K$-linear span of monomials $\mu\eta^*,$ where $\mu,\eta$ are paths in $A_n$ 
such that $w(\mu\eta^*)=(i,j-i,j).$ With thus defined grading, the above 
isomorphism tells us that $L_K(A_n)$ and $\mathbb{M}_n(K)$ are also graded 
isomorphic as $S$-graded rings.\\
\indent As we know, since $K$ is von Neumann regular, 
$\mathbb{M}_n(K)$ is von Neumann regular. In particular, for every homogeneous 
element $A$ of an $S$-graded ring $\mathbb{M}_n(K),$ there exists 
an element $A'\in\mathbb{M}_n(K),$ which is homogeneous, such that $A=AA'A.$ 
Hence, $L_K(A_n)$ is graded von Neumann regular as an $S$-graded ring. Note that 
$S$ is cancellative and that it is such that for every $s\in S,$ there exist 
$s^{-1}\in S$ and idempotent elements $e,$ $f\in S$ for which $es=sf=s,$ 
$fs^{-1}=s^{-1}e=s^{-1},$ $ss^{-1}=e$ and $s^{-1}s=f,$ in which case we say in
this article that $S$ satisfies the (LRI) condition (existence of left and right 
inverses). A cancellative groupoid $S$ needs to be such in order for an $S$-graded 
ring to be graded von Neumann regular (Proposition~\ref{proposition2}).\\
\indent Let $B$ be an arbitrary Brandt groupoid over $\mathbb{Z}.$ The 
given example triggers out the following natural questions.\\ 
\indent \emph{Given an arbitrary directed graph $E,$ which is not null, and a unital 
von Neumann regular ring $R,$ is it possible to observe $L_R(E)$ as a ring graded by 
$B$ so that $L_R(E)$ is $B$-graded von Neumann regular?}\\
\indent \emph{Conversely, if $L_R(E)$ is $B$-graded von Neumann regular, is $R$ 
von Neumann regular?}\\
\indent Theorem~\ref{theoremmain} answers these questions in the affirmative, thus
generalizing Theorem~\ref{theoremdl}. The Brandt groupoid gradings on Leavitt path 
algebras of the desired kind are explained in Section~\ref{gleavitt}. The approach for
proving Theorem~\ref{theoremmain}, inspired by that of \cite{rh}, is similar to the
approach taken in the proof of Theorem~\ref{theoremdl}. Namely, the result is first 
established for finite graphs, and then, for arbitrary graphs by observing the Leavitt 
path algebra as a direct limit of Leavitt path algebras over finite graphs.\\
\indent In category theoretic terminology, Brandt groupoids are connected small 
categories all of whose morphisms are invertible, which are also known simply as 
groupoids. However, in this article, as we have already pointed out, by a groupoid we 
mean a set with a binary operation, which is also known in the literature as magma. 
Let $K$ be a field and $E$ a directed graph. The Leavitt path algebra $L_K(E)$ has 
already been observed as a ring graded by a small category all of whose morphisms 
are invertible via partial actions. Namely, in \cite{gy}, the Leavitt path algebra 
$L_K(E)$ is realized as a partial skew groupoid ring 
$D(X)\star_\lambda\mathrm{G}(E),$ where $D(X)$ is a certain $K$-algebra and 
$\lambda$ is a partial action of the free path groupoid $\mathrm{G}(E)$ on $D(X).$ 
Under some restrictions, in \cite{bmp}, the $\mathrm{G}(E)$-graded von Neumann 
regularity of $L_K(E)$ is characterized in terms of $D(X).$ Namely, let $E$ be a 
connected graph and let the set of vertices $E^0$ be finite. Then, $L_K(E)$ is 
$\mathrm{G}(E)$-graded von Neumann regular if and only if $D(X)$ is 
von Neumann regular, provided that the characteristic of the field $K$ does not 
divide $|E^0|1_K.$\\
\indent One of the key notions used in obtaining Theorem~\ref{theoremdl} is the 
notion of a \emph{nearly epsilon-strongly group graded ring}, which is 
introduced in \cite{nop} as a generalization of an \emph{epsilon-strongly group 
graded ring} \cite{nop}, class of which, in turn, contains all strongly group graded 
unital rings. Let $G$ be a group with identity $e.$ According to Theorem~1.2 in 
\cite{dl}, we have that $R=\bigoplus_{g\in G}R_g$ is graded von Neumann regular if 
and only if $R$ is nearly epsilon-strongly graded and $R_e$ is von Neumann regular. 
As a corollary, one also obtains that an epsilon-strongly graded ring 
$R=\bigoplus_{g\in G}R_g$ is graded von Neumann regular if and only if $R_e$ is 
von Neumann regular (Corollary~3.11 in \cite{dl}). For results concerning rings 
graded by finite small categories all of whose morphisms are invertible, the reader is 
referred to \cite{nop1,bmp}.\\
\indent In order to prove Theorem~\ref{theoremmain}, and also, as one of the aims 
of this paper, we study the graded von Neumann regularity of rings which are not 
necessarily graded by a group or by a small category all of whose morphisms are
invertible, which is the content of Section~\ref{gradedvnr}. More precisely, we are 
interested in rings graded in the sense of the following definition, which, in particular, 
includes all the other gradings. 
\begin{definition}[\cite{ak2,ak}]\label{sgris} Let $R$
be a ring, and $S$ a partial groupoid, that is, a set with a
partial binary operation. Also, let $\{R_s\}_{s\in S}$ be a family of additive 
subgroups of $R,$ called \emph{components.} We say that
$R=\bigoplus_{s\in S}R_s$ is $S$-\emph{graded}
and $R$ \emph{induces} $S$ (or $R$ is an $S$-\emph{graded
ring inducing} $S$) if the following two conditions hold:
\begin{enumerate}
    \item[$i)$] $R_s R_t\subseteq R_{st}$ whenever $st$ is defined;
    \item[$ii)$] $R_s R_t\neq0$ implies that the product $st$ is
    defined.
\end{enumerate}
The set $H_R=\bigcup_{s\in S}R_s$ is called the \emph{homogeneous part of $R,$} 
and elements of $H_R$ are called \emph{homogeneous elements of $R.$}
\end{definition}
This definition applies to both associative and nonassociative rings. For the results on 
nonassociative rings graded by a set, we refer the reader to \cite{ek,vk,ajc} and 
references therein. Note that the associativity of an $S$-graded ring inducing $S$ 
does not imply the associativity in $S.$ As examples of $S$-graded rings inducing 
$S,$ let us mention a semidirect extension of a ring (in particular, the Dorroh 
extension), generalized matrix rings (in particular, the already given example of
$\mathbb{M}_n(K)$), every group or semigroup graded ring, see \cite{ak}, etc.\\
\indent We finish the article by applying Theorem~\ref{theoremmain} in order to 
show that the Leavitt path algebra over a unital von Neumann regular ring enjoys 
some properties which represent groupoid graded counterparts of those satisfied by 
von Neumann regular rings. In particular, if $R$ is von Neumann regular, then one 
type of the graded Jacobson radical of $L_R(E)$ is zero, and $L_R(E)$ is graded 
semiprime as a Brandt groupoid graded ring (Theorem~\ref{theoremgs}).
\section{Preliminaries}\label{preliminaries}
\subsection{Graded rings}
Let $R=\bigoplus_{s\in S}R_s$ be an $S$-graded ring inducing $S.$
\emph{The degree} $\deg(a)$ of a nonzero homogeneous element $a$ of $R$ is 
defined to be a unique $s\in S$ such that $a\in R_s.$ We define $0=\deg(0),$ and
we may without loss of generality assume that $0\in S$ since the zero element of
$R$ can be viewed as a component of $R.$ We may also assume that 
$S\setminus\{0\}=\{s\in S\ |\ R_s\neq0\}.$ Hence $R_0=0,$ the zero subring 
of $R.$ Throughout the article, we make $S$ a groupoid by putting $st=0$ for those 
pairs $(s,t)\in S\times S$ for which the product $st$ is not originally defined 
(in which case $R_sR_t=0$). Also, $s0=0s=0$ for every $s\in S.$ This is done for 
every $S$-graded ring inducing $S$ without further notice. We set 
$S^\times=S\setminus\{0\}.$ It is clear that
$R=\bigoplus_{s\in S}R_s=\bigoplus_{s\in S^\times}R_s.$
\begin{remark}\label{remarkassoc}
Note that for $s,$ $t,$ $u\in S^\times,$ if $R_sR_tR_u\neq0,$ then $(st)u=s(tu).$
In that case, as usual, we write this element as $stu.$
\end{remark}
Throughout the article, a groupoid $S$ with zero $0$ is said to be 
\emph{cancellative} if for $s,$ $t,$ $u\in S,$ each of the equalities 
$0\neq su=tu\in S$ or $0\neq us=ut\in S$ implies $s=t.$ Also, the set of
all idempotent elements of $S$ is denoted by $I(S).$ By $I(S)^\times$ we denote
the set $I(S)\setminus\{0\}.$\\
\indent If $R=\bigoplus_{s\in S}R_s$ is an $S$-graded ring inducing $S,$ then $R$ is 
said to be \emph{strongly graded} if $R_sR_t=R_{st}$ for every $s,$ $t\in S.$\\
\indent We also note that the notions of an $S$-graded ring inducing $S$ and of a 
graded ring studied in \cite{ha1,ha,agg} are equivalent.\\
\indent Let $R=\bigoplus_{s\in S}R_s$ be an $S$-graded ring inducing $S.$ A right 
(left, two-sided) ideal $I$ of $R$ is said to be \emph{homogeneous} if 
$I=\bigoplus_{s\in S}R_s\cap I.$ Also recall that if $I$ is a homogeneous ideal
(two-sided) of $R$ and $I_s=R_s\cap I,$ then $R/I=\bigoplus_{s\in S}R_s/I_s$ is an 
$S$-graded ring inducing $S$ \cite{ha,agg,ak3,ak}.\\
\indent If $R$ is an $S$-graded ring inducing $S$ and $R'$ an $S'$-graded ring
inducing $S',$ then a ring homomorphism $f: R\to R'$ is said to be 
\emph{homogeneous} \cite{ha1,ha,agg} if $f(H_R)\subseteq H_{R'}$ and if for $x,$ 
$y\in H_R$ such that $f(x),$ $f(y)\neq0,$ we have that $\deg(f(x))=\deg(f(y))$  
implies that $\deg(x)=\deg(y).$ It is easy to verify that the $S$-graded rings 
inducing $S,$ together with the homogeneous homomorphisms form a category. In 
particular, if $S$ is fixed, we denote such category by $S\operatorname{-RING}.$\\ 
\indent The category $S\operatorname{-RING}$ has arbitrary direct limits. Namely,
let $\{R_i\ |\ i\in I\}$ be a direct system of $S$-graded rings, where
$R_i=\bigoplus_{s\in S}(R_i)_s.$ Then it can be easily verified, like in the group
graded case, that $A=\underrightarrow{\lim}_iR_i$ is an $S$-graded ring with the
components $A_s=\underrightarrow{\lim}_i(R_i)_s.$\\
\indent The following lemma will be used in the proof of 
Theorem~\ref{theoremmain}. Its proof is similar to the group graded case (see
Proposition~5.2.14 in \cite{bk}) and therefore, omitted.
\begin{lemma}\label{lemmadlvnr}
Let $\{R_i\ |\ i\in I\}$ be a direct system of rings which are objects in 
$S\operatorname{-RING.}$ If $R_i$ is graded von Neumann regular as an object in
$S\operatorname{-RING}$ for every $i\in I,$ then the direct limit of 
$\{R_i\ |\ i\in I\}$ is graded von Neumann regular as an object in 
$S\operatorname{-RING}$ too.
\end{lemma}
\subsection{The graded Jacobson radical}
Throughout the article, the classical Jacobson radical of a ring $A$ is denoted as
usual by $J(A).$\\
\indent Let $R$ be an $S$-graded ring inducing $S$ and let us assume that $S$ is
cancellative. A homogeneous right ideal $I$ of $R$ is said to be a \emph{graded 
modular right ideal} \cite{ha1,ha} if there exists a homogeneous element $u\in R$ 
such that $ux-x\in I$ for every homogeneous element $x\in R.$ The cancellativity of 
$S$ gives that $\deg(u)$ is an idempotent element of $S,$ and that all such 
elements $u$ are of the same degree, which is referred to as \emph{the degree of 
$I.$}\\ 
\indent The \emph{graded Jacobson radical $J^g(R)$ of $R$} \cite{ha1,ha}
is the intersection of all graded maximal modular right ideals of $R.$ If $J^g(R)=0,$
then we say that $R$ is \emph{graded semisimple}. For the study
of other graded radicals of $S$-graded rings inducing $S,$ we refer the reader to
\cite{ak1,ak2,eig1,eig4,eig11,eig7,eig9} and references therein.\\
\indent Let $e$ be an idempotent element in $S^\times.$ There exists a one-to-one 
correspondence between the set of all graded maximal modular right ideals of $R$ of 
degree $e$ and the set of all maximal modular right ideals of the ring $R_e,$ given 
by $I\mapsto I\cap R_e,$ see \cite{ha1,ha}. As a corollary, one obtains that 
$J^g(R)=\bigoplus_{s\in S}I_s,$ where 
$I_s=\{x\in R_s\ |\ (\forall e\in I(S))\ xH_R\cap R_e\subseteq J(R_e)\}.$ It is then 
easy to verify the following statements. These represent one of the key 
ingredients we use in a characterization of specific kinds of graded von Neumann 
regular $S$-graded rings inducing $S$ (Proposition~\ref{propositiongs} and
Theorem~\ref{corollary}).
\begin{theorem}[\cite{ha1,ha}]\label{theoremj}
Let $R=\bigoplus_{s\in S}R_s$ be an $S$-graded ring inducing $S,$ where $S$ is
cancellative. Then:
\begin{itemize}
\item[$a)$] $J^g(R)\cap R_e=J(R_e)$ for all $e\in I(S);$
\item[$b)$] $J^g(R)=0,$ that is, $R$ is graded semisimple, if and only if the 
following two conditions are satisfied:
\begin{itemize}
\item[$i)$] $J(R_e)=0$ for every $e\in I(S),$ that is, each ring component of $R$ 
is semisimple;
\item[$ii)$] For every nonzero homogeneous element $x\in R$ there exists a
homogeneous element $y\in R$ such that $xy$ is a nonzero homogeneous element
of an idempotent degree;
\end{itemize}
\item[$c)$] Let $s\in S.$ If $R_s$ is not contained in the graded Jacobson 
radical $J^g(R),$ then there exist elements $e,$ $f\in I(S)$ and an element 
$s^{-1}\in S$ such that $es=sf=s,$ $fs^{-1}=s^{-1}e=s^{-1},$ $ss^{-1}=e,$ and 
$s^{-1}s=f.$
\end{itemize} 
\end{theorem}
\begin{remark}
Note that there are graded rings $R$ in which $J(R)=0$ but $J^g(R)\neq0.$ For
instance, if $K$ is a field, the $\mathbb{N}_0$-graded polynomial ring 
$K[x]=\bigoplus_{n\in\mathbb{N}_0}Kx^n$ is such a ring, where $\mathbb{N}_0$
is the additive monoid of nonnegative integers.
\end{remark}
\subsection{Von Neumann regular rings}
Let us recall that a ring $R$ is said to be \emph{von Neumann regular} if for every 
$x\in R$ we have that $x\in xRx.$ A ring $R$ is said to be \emph{s-unital} if, for 
every $x\in R,$ there exist $e,$ $e'\in R$ such that $ex=x=xe'.$ It is clear
that every von Neumann regular ring is s-unital (Proposition~2.1 in \cite{dl}). The 
following proposition is established in \cite{dl}, and represents a generalization of a 
well-known characterization of von Neumann regularity for unital rings 
(Theorem~1.1 in \cite{krg}) to s-unital rings.
\begin{proposition}[Proposition~2.2 in \cite{dl}]\label{propositionvn}
Let $R$ be an s-unital ring. Then the following statements are equivalent:
\begin{itemize}
\item[$a)$] $R$ is von Neumann regular;
\item[$b)$] Every principal right (left) ideal of $R$ is generated by an idempotent
element;
\item[$c)$] Every finitely generated right (left) ideal of $R$ is generated by an
idempotent element.
\end{itemize}
\end{proposition}
Since the Jacobson radical of a ring does not contain nonzero idempotent elements,
the following corollary is immediate.
\begin{corollary}[cf.~Corollary~1.2 in \cite{krg}]\label{corollaryj}
Let $R$ be an s-unital ring. If $R$ is von Neumann regular, then the Jacobson
radical $J(R)$ is zero.
\end{corollary}
\section{A Brandt groupoid graded von Neumann regularity of Leavitt path algebras}\label{gleavitt}
\indent Let us recall first the notion of the \emph{Leavitt path algebra} of a
directed graph over a unital ring.\\ 
\indent A \emph{directed graph} $E=(E^0,E^1,\mathrm{r},\mathrm{s})$ consists of 
two sets $E^0,$ $E^1$ and mappings $\mathrm{r},$ $\mathrm{s}:$ $E^1\to E^0.$ 
When there are more directed graphs observed at the same time, the mappings 
$\mathrm{r}$ and $\mathrm{s}$ of the graph $E$ are denoted by 
$\mathrm{r}_E$ and $\mathrm{s}_E,$ respectively. Elements of $E^0$ are called 
\emph{vertices} and elements of $E^1$ are called \emph{edges}. We say that $E$ 
is \emph{finite} if the cardinal number of $E^0$ is finite.\\
\indent A vertex $v$ for which $\mathrm{s}^{-1}(v)$ is empty is called a 
\emph{sink}, while a vertex $v'$ for which $\mathrm{r}^{-1}(v')$ is empty is called 
a \emph{source}. A vertex $v$ for which the cardinality of $\mathrm{s}^{-1}(v)$ is 
infinite, is called an \emph{infinite emitter}. A vertex $v,$ that is neither a sink nor 
an infinite emitter, is called a \emph{regular vertex}.\\
\indent The set of all sinks in $E$ is denoted by $\mathrm{Sink}(E),$ while the set
of all regular vertices in $E$ is denoted by $\mathrm{Reg}(E).$\\
\indent A \emph{path} $\mu$ in a graph $E$ is a sequence of edges 
$\mu=\alpha_1\dots\alpha_k$ such that 
$\mathrm{r}(\alpha_i)=\mathrm{s}(\alpha_{i+1}),$ $i=1,\dots,k-1,$
where $k\in\mathbb{N}.$ The \emph{length} of $\mu$ is equal to $k.$ We also 
put $\mathrm{s}(\mu):=\mathrm{s}(\alpha_1),$ and call it the \emph{source of 
$\mu$} and $\mathrm{r}(\mu):=\mathrm{r}(\alpha_k)$ is called the \emph{range 
of $\mu$}. Any vertex $v\in E^0$ can be observed as a trivial path of length zero 
with $\mathrm{s}(v)=\mathrm{r}(v)=v.$ By $E^k$ we denote the set of all paths of 
length $k,$ where $k\in\mathbb{N}_0,$ and we set 
$E^*:=\bigcup_{k\in\mathbb{N}_0}E^k,$ where, as usual, $\mathbb{N}_0$ denotes
the set of nonnegative integers.\\
\indent Let $E$ and $F$ be directed graphs. Recall also that by a \emph{graph 
homomorphism} $\psi: E\to F$ we mean a pair of mappings 
$(\psi^0: E^0\to F^0, \psi^1: E^1\to F^1)$ for which we have that 
$\mathrm{s}(\psi^1(\alpha))=\psi^0(\mathrm{s}(\alpha))$ and 
$\mathrm{r}(\psi^1(\alpha))=\psi^0(\mathrm{r}(\alpha))$ for every 
$\alpha\in E^1.$
\begin{definition}[Definition~1.6.2 in \cite{aasm}]
The category $\mathcal{G}$ is defined as the category with pairs $(E,X),$ where 
$E$ is a directed graph and $X\subseteq\mathrm{Reg}(E),$ as objects, and if 
$(F,Y)$ and $(E,X)$ are objects of $\mathcal{G},$ then $\psi=(\psi^0,\psi^1)$ is a 
morphism in $\mathcal{G}$ if the following conditions are satisfied:
\begin{itemize}
\item[$a)$] $\psi: F\to E$ is a graph homomorphism such that $\psi^0$ and
$\psi^1$ are injective;
\item[$b)$] $\psi^0(Y)\subseteq X;$
\item[$c)$] For every $v\in Y,$ the restriction 
$\psi^1: \mathrm{s}^{-1}_F(v)\to \mathrm{s}^{-1}_E(\psi^0(v))$ is a bijection.
\end{itemize}
A morphism $\psi$ is said to be \emph{complete} if, for every $v\in F^0,$ if
$\psi^0(v)\in X$ and $\mathrm{s}^{-1}_F(v)\neq\emptyset,$ then $v\in Y.$
\end{definition}
\begin{definition}[Definition~1.6.8 in \cite{aasm}]
Let $E=(E^0,E^1,\mathrm{r},\mathrm{s})$ be a directed graph and 
$X\subseteq\mathrm{Reg}(E).$ A subgraph $F$ of $E$ is said to be
\emph{$X$-complete} if the inclusion mapping 
$(F,\mathrm{Reg}(F)\cap X)\to(E,X)$ is a complete morphism in the category
$\mathcal{G}.$
\end{definition}
\begin{definition}\label{definitionlpa}
For a directed graph $E=(E^0,E^1,\mathrm{r},\mathrm{s}),$ a unital ring $R,$ 
and $X\subseteq\mathrm{Reg}(E),$ the \emph{Cohn path algebra of $E$ with 
respect to $X$} \cite{ag}, denoted by $C_R^X(E),$ is the free algebra 
generated by the sets $\{v\ |\ v\in E^0\},$ $\{\alpha\ |\ \alpha\in E^1\}$ and 
$\{\alpha^*\ |\ \alpha\in E^1\}$ with the coefficients in $R,$ subject to the
relations:
\begin{itemize}
\item[$1)$] $v_iv_j=\delta_{i,j}v_i$ for every $v_i,$ $v_j\in E^0;$
\item[$2)$] $\mathrm{s}(\alpha)\alpha=\alpha\mathrm{r}(\alpha)=\alpha$ and 
$\mathrm{r}(\alpha)\alpha^*=\alpha^*\mathrm{s}(\alpha)=\alpha^*$ for all 
$\alpha\in E^1;$
\item[$3)$] $\alpha^*\alpha'=\delta_{\alpha,\alpha'}\mathrm{r}(\alpha)$ for all 
$\alpha,$ $\alpha'\in E^1;$
\item[$4)$] $\sum_{\alpha\in E^1, \mathrm{s}(\alpha)=v}\alpha\alpha^*=v$ for 
every $v\in X.$
\end{itemize} 
We let $R$ commute with the set of generators of $C_R^X(E).$ If 
$X=\mathrm{Reg}(E),$ then $C_R^X(E)$ is called the \emph{Leavitt path algebra} 
\cite{aap,amp,rh} of $E$ over $R,$ and is denoted by $L_R(E).$
\end{definition}
The elements $\alpha^*,$ where $\alpha\in E^1,$ are called the 
\emph{ghost edges}. The set of all ghost edges is denoted by $(E^1)^*.$\\
\indent Clearly, $L_R(E)$ is with unity $1$ if and only if $E^0$ is finite, and in that 
case, $1=\sum_{v\in E^0}v$ (see for instance Lemma~1.2.12 in \cite{aasm}).\\
\indent If $\mu=\alpha_1\dots\alpha_k,$ where $\alpha_i\in E^1,$ is a path in $E,$
that is, if $\mu\in E^*,$ then, without further notice, we observe $\mu$ as an 
element of $C_R^X(E).$ Also, if $\mu=\alpha_1\dots\alpha_k\in E^*,$ then by 
$\mu^*$ we denote the element $\alpha_k^*\dots\alpha_1^*\in C_R^X(E).$ We 
moreover put $v^*=v$ for all $v\in E^0.$ According to the condition $3)$ of 
Definition~\ref{definitionlpa}, any word in the generators 
$\{v, \alpha, \alpha^*\ |\ v\in E^0, \alpha\in E^1\}$ in $C_R^X(E)$ can be written 
as $\mu\eta^*,$ where $\mu$ and $\eta$ are paths in $E.$ Elements of the form 
$\mu\eta^*,$ where $\mu,$ $\eta\in E^*,$ are called \emph{monomials}.
\subsection{Brandt groupoid gradings on Leavitt path algebras}
Let $E=(E^0,E^1,\mathrm{r},\mathrm{s})$ be a directed graph, $R$ a unital ring, 
$X\subseteq\mathrm{Reg}(E),$ and $C_R^X(E)$ the Cohn path algebra of $E$ with 
respect to $X$ with coefficients in $R.$\\
\indent Let $B$ be a Brandt groupoid with the set of all idempotent elements 
denoted by $I(B).$ Let us recall (see for instance \cite{cp}), a Brandt groupoid is a 
partial groupoid $B$ which satisfies the following axioms:
\begin{itemize}
\item[(B1)] If $st=u$ $(s,t,u\in B),$ then each of the three elements $s,$ $t,$ $u$
is uniquely determined by the other two;
\item[(B2)] Let $s,$ $t,$ $u\in B.$
\begin{itemize}
\item[(i)] If $st$ and $tu$ are defined, so are $(st)u$ and $s(tu),$ and 
$(st)u=s(tu);$
\item[(ii)] If $st$ and $(st)u$ are defined, so are $tu$ and $s(tu),$ and
$s(tu)=(st)u;$
\item[(iii)] If $tu$ and $s(tu)$ are defined, so are $st$ and $(st)u,$ and
$(st)u=s(tu);$
\end{itemize}
\item[(B3)] For every element $s\in B,$ there exist unique elements $e,$ $f,$
$s'\in B$ such that $es=sf=s$ and $s's=f;$
\item[(B4)] If $e,$ $f\in I(B),$ then there exists $s\in B$ such that $es=sf=s.$
\end{itemize}
\indent Note that $B$ satisfies the following condition:
\begin{itemize}
\item[(LRI)] For every $s\in B$ there exist $s^{-1}\in B$ and $e,$ $f\in I(B)$ such 
that $es=sf=s,$ $fs^{-1}=s^{-1}e=s^{-1},$ $ss^{-1}=e$ and $s^{-1}s=f.$
\end{itemize}
Let us also recall that $B$ is isomorphic to a partial groupoid $(M(G,I),\circ),$ for 
some group $G,$ and some set $I.$ Here, $M(G,I)$ consists of triples $(i,g,j)$ 
$(g\in G, i, j\in I),$ and $\circ$ is defined by $(i,g,j)\circ(k,h,l)=(i,gh,l)$ if $j=k.$
Without loss of generality, we may assume that $|I|=|E^0|.$\\
\indent Throughout this section, let $S=B\cup\{0\}$ and let us make $S$ a groupoid 
by setting $st=0$ for all $s,$ $t\in B$ for which $st$ is not defined in $B,$ and 
$0s=s0=0$ for all $s\in B.$ As it is well-known, $S$ is then a semigroup, known as 
a \emph{Brandt semigroup} (see \cite{cp}). We also write $S^\times$ instead 
of $B.$ Let $I(S)=I(B)\cup\{0\}.$ Of course, the semigroup 
$S$ also satisfies the (LRI) condition.\\
\indent Since $S$ is cancellative, note that idempotents $e$ and 
$f,$ and element $s^{-1}$ from the condition (LRI) are unique.\\
\indent We define a \emph{weight mapping} 
$w:E^*\cup\{\mu^*\ |\ \mu\in E^*\}\to B$ such that
\begin{eqnarray}\nonumber
w|_{E^0}:E^0&\to&I(B),\\
\nonumber w|_{E^1}:E^1&\to&B,\\
\nonumber w|_{(E^1)^*}:(E^1)^*&\to&B,
\end{eqnarray}  
subject to the following rules:
\begin{itemize}
\item[$w1)$] If $\alpha\in E^1,$ then $w(\alpha)$ is such that 
$w(\mathrm{s}(\alpha))w(\alpha)=w(\alpha)=w(\alpha)w(\mathrm{r}(\alpha)).$ 
\item[$w2)$] $w(\alpha^*)=(w(\alpha))^{-1}$ for every $\alpha\in E^1.$
\end{itemize}
\indent Note that the weight mapping from the example given in the Introduction, is
in lines with the listed rules.\\
\indent Now, let $w$ be a weight mapping defined with respect to the rules $w1)$
and $w2).$ For $s\in S^\times$ let
\[(C_R^X(E))_s=\{\sum_ir_i\mu_i\eta_i^*\ |\ r_i\in R, \mu_i, 
\eta_i\in E^*, \mathrm{r}(\mu_i)=\mathrm{r}(\eta_i), w(\mu_i\eta_i^*)=s\},\]
where the indicated sums are finite. Moreover, for the zero element $0$ of $S,$ 
let $(C_R^X(E))_0=\{0\}.$
\begin{definition}
Let $R=\bigoplus_{s\in T}R_s$ be a $T$-graded ring inducing $T,$ where 
$T$ is cancellative and with the property (LRI). Then an involution mapping 
$^*:R\to R,$ that is, a mapping for which $(x^*)^*=x$ for every $x\in R,$ is said to 
be \emph{anti-graded} if $(R_s)^*=R_{s^{-1}}$ for all $s\in T.$
\end{definition}
\begin{proposition}\label{propositiongrading}
The Cohn path algebra $C_R^X(E)$ with respect to $X$ is an $S$-graded ring 
$C_R^X(E)=\bigoplus_{s\in S}(C_R^X(E))_s$ with an anti-graded involution with 
respect to the given weight mapping $w.$ In particular, the Leavitt path algebra 
$L_R(E)=C_R^{\mathrm{Reg}(E)}(E)$ is an $S$-graded ring with an anti-graded 
involution. 
\end{proposition}
\begin{proof}
Let $R\langle Y\rangle$ be a free algebra over $R$ generated by the set
$Y=E^0\cup E^1\cup (E^1)^*.$ Then $R\langle Y\rangle$ is clearly an $S$-graded
ring with respect to $w.$ By definition, $C_R^X(E)$ is the factor algebra 
$(R\langle Y\rangle)/I,$ where $I$ is an ideal generated by the elements of the form:
\begin{enumerate}
\item $v_iv_j-\delta_{i,j}v_i$ for every $v_i,$ $v_j\in E^0;$
\item $\mathrm{s}(\alpha)\alpha-\alpha,$ $\alpha\mathrm{r}(\alpha)-\alpha,$  
$\mathrm{r}(\alpha)\alpha^*-\alpha^*,$ $\alpha^*\mathrm{s}(\alpha)-\alpha^*$ 
for all $\alpha\in E^1;$
\item $\alpha^*\alpha'-\delta_{\alpha,\alpha'}\mathrm{r}(\alpha)$ for all $\alpha,$ 
$\alpha'\in E^1;$
\item $v-\sum_{\alpha\in E^1, \mathrm{s}(\alpha)=v}\alpha\alpha^*$ for every 
$v\in X.$
\end{enumerate}
Now, all of these elements are homogeneous with respect to the given weight 
mapping $w.$ Therefore, $I$ is a homogeneous ideal of $R\langle Y\rangle.$ It 
follows that $C_R^X(E)$ is $S$-graded. Also, every element of $C_R^X(E)$ is a 
finite $R$-linear combination of monomials of the form $\mu\eta^*,$ where $\mu,$
$\eta\in E^*,$ and $\mathrm{r}(\mu)=\mathrm{r}(\eta).$ Hence, 
$C_R^X(E)=\bigoplus_{s\in S}(C_R^X(E))_s.$ For 
$\sum_ir_i\mu_i\eta_i^*\in C_R^X(E),$ let us define
$\left(\sum_ir_i\mu_i\eta_i^*\right)^*=\sum_ir_i\eta_i\mu_i^*.$ Then 
$^*:C_R^X(E)\to C_R^X(E)$ is an involution mapping (see also \cite{mt}).
Let $s\in S$ and $0\neq x\in(C_R^X(E))_s.$ Also, let $e\in I(S)$ be such
that $ss^{-1}=e.$ Now, $x$ is a finite sum 
$\sum_ir_i\mu_i\eta_i^*,$ where $r_i\in R,$ and $\mu_i,$ $\eta_i\in E^*$ are 
such that $\mathrm{r}(\mu_i)=\mathrm{r}(\eta_i)$ and 
$w(\mu_i\eta_i^*)=s$ for every $i.$ Then 
$x^*=\sum_ir_i\eta_i\mu_i^*.$ Since $S$ is cancellative, and 
$\mathrm{s}(\mu_i)\mu_i\eta_i^*=\mu_i\eta_i^*,$ it follows that
$w(\mathrm{s}(\mu_i))=e.$ Also,
since $(\mu_i\eta_i^*)(\eta_i\mu_i^*)=\mu_i\mu_i^*,$ by the 
definition of $w,$ we get that 
$w(\mu_i\eta_i^*)w(\eta_i\mu_i^*)=w(\mu_i)w(\mu_i^*)
=w(\mathrm{s}(\mu_i))=e.$
Hence, $w(\eta_i\mu_i^*)=s^{-1}.$ Thus, $(C_R^X(E))_s^*=(C_R^X(E))_{s^{-1}}$ 
for every $s\in S,$ that is, $C_R^X(E)$ is equipped with an anti-graded involution. 
\end{proof}
\begin{remark}
Note that since 
$C_R^X(E)=\bigoplus_{s\in B}(C_R^X(E))_s=\bigoplus_{s\in S}(C_R^X(E))_s,$ we
may observe $C_R^X(E)$ as a ring graded by a Brandt groupoid $B=S^\times.$
\end{remark}
\begin{remark}\label{remark1}
In a particular case when $S^\times=\mathbb{Z}$ is the additive group of integers, 
one obtains the most explored canonical $\mathbb{Z}$-grading of $C_R^X(E)$
(see \cite{dl}) by defining $w(v)=0$ for every $v\in E^0,$ and $w(\alpha)=1$ 
($w(\alpha^*)=-1$) for every $\alpha\in E^1.$
\end{remark}
Let us assume now that $w,$ along with $w1)$ and $w2),$ moreover satisfies the 
following rules:
\begin{itemize}
\item[$w3)$] If $e$ is an idempotent element of $B,$ then every edge, which starts 
from the vertex of weight $e,$ is of the same weight.
\item[$w4)$] If $e$ is an idempotent element of $B,$ then every edge, which ends 
in the vertex of weight $e,$ is of the same weight.
\end{itemize}
\begin{remark}\label{remarkZ}
Note that in the case of a connected graph which contains loops, we have that every
vertex is of the same weight and that every edge is of the same weight. Therefore,
in that case, if we take $B=M(\mathbb{Z},I)$ to be a Brandt groupoid over 
$\mathbb{Z},$ then we obtain the canonical $\mathbb{Z}$-grading of $L_R(E)$ by 
putting $w(v)=(i,0,i)$ for every $v\in E^0$ and $w(\alpha)=(i,1,i)$ 
($w(\alpha^*)=(i,-1,i)$) for every $\alpha\in E^1,$ for some $i\in I.$
\end{remark}
\indent Of course, there are many ways to define $w$ which satisfies the rules
$w1)$-$w4).$ For instance, let $v\in E^0$ be a vertex for which 
$\mathrm{s}^{-1}(v)\neq\emptyset.$ Take $e\in I(S)^\times$ and put $w(v)=e.$\\
\indent Case 1. $v$ belongs to a component which contains loops. Then, 
by Remark~\ref{remarkZ}, the component which contains vertex $v$ is 
such that every vertex is of the same weight $e,$ and every edge is of the same 
weight $s$ such that $es=se=s.$\\
\indent Case 2. $v$ belongs to a component which contains no loops. Let 
$e\neq f\in I(S)^\times,$ and let $s\in S$ be such that $es=sf=s.$ Then, for every 
edge $\alpha\in\mathrm{s}^{-1}(v),$ we put $w(\alpha)=s$ and 
$w(\mathrm{r}(\alpha))=f.$ Further, for each $\alpha\in\mathrm{s}^{-1}(v),$ if 
$\mathrm{s}^{-1}(\mathrm{r}(\alpha))\neq\emptyset,$ we put $w(\gamma)=t$ 
for every $\gamma\in\mathrm{s}^{-1}(\mathrm{r}(\alpha)),$ where 
$t\in S$ is such that $ft=t,$ and so on. Note that $st\neq0.$\\
\indent Of course, if we start with $f=e,$ it leads us to the weight mapping from 
the Case 1.\\
\indent If $v'$ is a vertex for which $\mathrm{r}^{-1}(v')\neq\emptyset,$ then we 
proceed analogously in the opposite direction with the given rules in mind. For 
instance, let $w(v')=e'.$ Again, if $v'$ belongs to a component with loops, then
each vertex of that component is of weight $e'$ and each edge of that component  
is of weight $s',$ where $s'$ is such that $s'e'=e's'=s'.$ Assume that $v'$ belongs
to a component without loops. Take $e'\neq f'\in I(B).$ Then, for every 
$\beta\in\mathrm{r}^{-1}(v),$ we put $w(\beta)=\bar{s},$ where $\bar{s}$ is
such that $\bar{s}e'=f'\bar{s}=\bar{s},$ etc.\\
\indent Finally, we put $w(\alpha^*)=(w(\alpha))^{-1}$ for every 
$\alpha\in E^1.$\\
\indent In view to Remark~\ref{remark1}, we introduce the following notion.
\begin{definition}\label{definitionc}
Let $S=B\cup\{0\}$ be a Brandt semigroup over the additive group of integers 
$\mathbb{Z},$ that is, $B=M(\mathbb{Z},I),$ for some set $I.$ Let 
$w:E^*\cup\{\mu^*\ |\ \mu\in E^*\}\to B$ be the weight mapping which satisfies 
the rules $w1)$-$w4),$ and it is such that the edges have the weights of the form 
$(i,1,j)\in B$ ($(j,-1,i)$ for the corresponding ghost edges) and vertices have the 
weights of the form $(i,0,i).$ Then an $S$-graded algebra $C_R^X(E),$ with respect 
to $w,$ in the sense of Proposition~\ref{propositiongrading}, is said to be 
\emph{canonically $S$-graded}.
\end{definition}
We now state the result which characterizes canonically $S$-graded von Neumann 
regularity of the Leavitt path algebras of arbitrary non-null directed graphs in terms 
of von Neumann regularity of the coefficient unital rings.
\begin{theorem}\label{theoremmain}
Let $R$ be a unital ring and let $E$ be a directed graph distinct from the null graph. 
If $S$ is a Brandt semigroup over the additive group of integers, then the Leavitt 
path algebra $L_R(E)$ is graded von Neumann regular as a canonically $S$-graded 
ring if and only if $R$ is von Neumann regular.
\end{theorem}
In particular, since a field $K$ is von Neumann regular, then $L_K(E)$ is 
graded von Neumann regular as a canonically $S$-graded ring. Therefore, by 
Remark~\ref{remark1}, as a corollary to Theorem~\ref{theoremmain}, one obtains 
Theorem~\ref{theoremrh}, as well as Theorem~\ref{theoremdl}.\\
\indent Also, following the setting from Remark~\ref{remarkZ}, in the 
case $E$ is a connected graph with loops, Theorem~\ref{theoremmain} 
coincides with Theorem~\ref{theoremdl}.  
\begin{remark}
As it is noted in Remark~4.7 in \cite{dl}, in case $E$ is the null graph, the Leavitt 
path algebra is the zero ring over any ring $R.$ Hence, $L_R(E)$ is then
graded von Neumann regular as an $S$-graded ring for any $R,$ which means that
Theorem~\ref{theoremmain} does not hold for the null graphs.
\end{remark}
\begin{remark}
Theorem~\ref{theoremmain} cannot be extended to arbitrary gradings by Brandt 
semigroups. Namely, let $S=M(\mathbb{Z},I)$ be a Brandt semigroup over 
$\mathbb{Z},$ and let $E=R_1$ be the graph with one vertex $v$ and one loop 
$\alpha.$ Also, let $K$ be a field, and let us put $w(v)=w(\alpha)=e=(i,0,i),$ for 
some $i\in I.$ Note that $w$ satisfies $w1)$ and $w2).$ Then, $L_K(R_1)$ is 
trivially $S$-graded in the sense of Proposition~\ref{propositiongrading}: 
$(L_K(R_1))_e=L_K(R_1),$ and $(L_K(R_1))_s=0$ for all $e\neq s\in S.$ 
However, as it is well-known, $L_K(R_1)\cong K[x,x^{-1}],$ and $K[x,x^{-1}]$ is 
not a von Neumann regular ring.
\end{remark}
\section{Groupoid graded von Neumann regular rings}\label{gradedvnr}
In this section, rings are graded in the sense of Definition~\ref{sgris}.
\begin{definition}
Let $R$ be an $S$-graded ring inducing $S.$ We say that $R$ is 
\emph{graded von Neumann regular} if $x\in xRx$ for every $x\in H_R.$
\end{definition}
The following characterization, as in the case of group graded rings, is clear, but we
include its proof for the sake of completeness.
\begin{proposition}\label{proposition1}
Let $R=\bigoplus_{s\in S}R_s$ be an $S$-graded ring inducing $S,$ and let $S$ be
cancellative. Then $R$ is graded von Neumann regular if and only if for every 
$x\in H_R$ there exists $y\in H_R$ such that $x=xyx.$
\end{proposition}
\begin{proof}
The `if' part is obvious by the definition of a graded von Neumann regular ring. Now, 
let us assume that $R$ is graded von Neumann regular, and let $x\in H_R.$ Without 
loss of generality, we may assume that $x\neq0.$ Since $R$ is graded von Neumann
regular, there exists $y\in R$ such that $x=xyx.$ Let $y=\sum_{s\in S}y_s$ be
a unique homogeneous decomposition of $y.$ Then $x=xyx=\sum_{s\in S}xy_sx.$
By the hypothesis, $S$ is cancellative. Therefore, if $s$ and $t$ are distinct
elements of $S$ such that $R_{\deg(x)}R_sR_{\deg(x)},$ 
$R_{\deg(x)}R_tR_{\deg(x)}\neq0,$ it follows by Remark~\ref{remarkassoc} that 
$\deg(x)s\deg(x)\neq\deg(x)t\deg(x).$ Hence, from $x=\sum_{s\in S}xy_sx$ we 
get that there exists $s\in S$ such that $x=xy_sx,$ which concludes the proof.
\end{proof}
\begin{lemma}\label{lemma1}
Let $R=\bigoplus_{s\in S}R_s$ be an $S$-graded ring inducing $S,$ and let us
suppose that $S$ is cancellative. If $R$ is graded von Neumann regular, then
$R_e$ is von Neumann regular for every $e\in I(S).$
\end{lemma}
\begin{proof}
Let $e$ be an arbitrary idempotent element from $S^\times,$ and 
$0\neq x\in R_e.$ Since $R$ is graded von Neumann regular, 
Proposition~\ref{proposition1} implies that there exists $y\in H_R$ such that 
$x=xyx.$ Now, since $x\neq0,$ we get by Remark~\ref{remarkassoc} that 
$e=e\deg(y)e.$ By the cancellativity of $S,$ it follows that $\deg(y)=e.$ Hence, 
$R_e$ is a von Neumann regular ring.
\end{proof}
It is known from the group graded case that the reverse statement of the previous
lemma does not hold (see for instance Example~3.1 in \cite{dl}).\\
\indent The following notion generalizes the notion of a nearly epsilon-strongly group 
graded ring from \cite{no} to the case of $S$-graded rings inducing $S.$
\begin{definition}\label{definitiones}
Let $R=\bigoplus_{s\in S}R_s$ be an $S$-graded ring inducing $S,$ where $S$
is cancellative. We say that $R$ is \emph{nearly epsilon-strongly graded} if the 
following conditions are satisfied:
\begin{itemize}
\item[$i)$] $S$ satisfies $(\mathrm{LRI});$
\item[$ii)$] For every $s\in S$ and $x\in R_s$ there exist 
$\epsilon(x)\in R_sR_{s^{-1}}$ and $\epsilon'(x)\in R_{s^{-1}}R_s$ such that
$\epsilon(x)x=x=x\epsilon'(x).$
\end{itemize}
\end{definition}
\begin{remark}\label{remarknesg}
If $R=\bigoplus_{s\in S}R_s$ is nearly epsilon-strongly graded, we note that 
$R_eR_f=0$ for all distinct $e,$ $f\in I(S).$ Namely, let $e,$ $f\in I(S)^\times$ and 
$e\neq f.$ If $R_eR_f\neq0,$ then there exist $x\in R_e$ and $y\in R_f$ such that
$xy\neq0,$ and $\deg(xy)=ef.$ Since $R$ is nearly epsilon-strongly graded, there 
exists $\epsilon(x)\in R_e$ such that $\epsilon(x)x=x.$ Hence, 
$\epsilon(x)xy=xy\neq0.$ So, we get that $eef=ef.$ Since $S$ is cancellative, it 
follows that $ef=f.$ Hence, $e=f,$ a contradiction. Therefore, $R_eR_f=0$ for all 
distinct $e,$ $f\in I(S).$
\end{remark}
\begin{example}\label{examplenesg}
Let $R$ be an s-unital ring, and let $\mathbb{M}_2(R)=\left(%
\begin{array}{cc}
  R & R \\
  R & R \\
\end{array}%
\right)$ be the ring of $2\times 2$ matrices with the usual matrix addition and
multiplication over $R.$ Also, let $(\mathbb{M}_2(R))_{(i,j)}$ be the subset of 
$\mathbb{M}_2(R)$ which consists of matrices with entries from $R$ in the $(i,j)$ 
position and zeroes elsewhere, where $i,j=1,2.$ Then, like in the example from 
Section~\ref{introduction}, 
$\mathbb{M}_2(R)=\bigoplus_{i,j=1,2}(\mathbb{M}_2(R))_{(i,j)}$ is a 
strongly $S$-graded ring inducing $S,$ where $S^\times=\{(i,j)\ |\ i,j\in\{1,2\}\},$
with respect to $(i,j)(k,l)=\delta_{j,k}(i,l),$ for all $i,$ $j,$ $k,$ $l=1,2$ 
(see \cite{ak, ak2}). Namely, since $R$ is s-unital, we have
\[(\mathbb{M}_2(R))_{(i,j)}(\mathbb{M}_2(R))_{(k,l)}=
\begin{cases}(\mathbb{M}_2(R))_{(i,l)} & \quad
\mathrm{if}\quad j=k;\\
O & \quad \mathrm{otherwise,}
\end{cases}\]
for all $i,$ $j,$ $k,$ $l=1,2,$ where $O$ denotes the zero matrix. Also, $S$ is 
cancellative, satisfies (LRI), and, of course, $S^\times$ is not a group. Let 
$(i,j)\in S^\times,$ and let $X\in(\mathbb{M}_2(R))_{(i,j)}$ be a nonzero matrix 
with a nonzero entry $x_{(i,j)}\in R.$ Since $R$ is $s$-unital, there exist $e_{(i,i)},$ 
$e_{(j,j)}\in R$ such that $e_{(i,i)}x_{(i,j)}=x_{(i,j)}=x_{(i,j)}e_{(j,j)}.$ Now, let 
$\epsilon(X)\in(\mathbb{M}_2(R))_{(i,i)}=(\mathbb{M}_2(R))_{(i,j)}
(\mathbb{M}_2(R))_{(j,i)}$ be a nonzero matrix with a nonzero entry $e_{(i,i)}$ and
let $\epsilon'(X)\in(\mathbb{M}_2(R))_{(j,j)}=(\mathbb{M}_2(R))_{(j,i)}
(\mathbb{M}_2(R))_{(i,j)}$ be a nonzero matrix with a nonzero entry $e_{(j,j)}.$ 
Then $\epsilon(X)X=X=X\epsilon'(X).$ Hence, $\mathbb{M}_2(R)$ is a nearly epsilon
strongly $S$-graded ring inducing $S.$
\end{example}
The following result justifies the usage of the (LRI) condition in the first place. 
\begin{proposition}\label{proposition2}
Let $R=\bigoplus_{s\in S}R_s$ be an $S$-graded ring inducing $S,$ and let us
assume that $S$ is cancellative. If $R$ is graded von Neumann regular, then $R$ is
nearly epsilon-strongly graded.
\end{proposition}
\begin{proof}
We prove that the conditions $i)$ and $ii)$ of Definition~\ref{definitiones} are 
satisfied. Let $s\in S^\times$ and $0\neq x\in R_s.$ Since $R$ is graded von 
Neumann regular, by Proposition~\ref{proposition1} there exist $t\in S$ and 
$y\in R_t$ such that $x=xyx.$ Then $s=sts$ by Remark~\ref{remarkassoc}. It 
follows that $e:=st\in I(S)^\times$ and that $f:=ts\in I(S)^\times.$ Hence, $es=s$ 
and $sf=s.$ Since $S$ is cancellative, we get that $t=tst.$ Hence, $t=tst=te=ft.$ 
Therefore, $i)$ holds true for $s^{-1}:=t.$ Now, 
$\epsilon(x):=xy\in R_sR_t$ is such that $x=(xy)x=\epsilon(x)x,$ and 
$\epsilon'(x):=yx\in R_tR_s$ is such that $x=x(yx)=x\epsilon'(x).$ Therefore, $ii)$ 
holds true as well. Thus, $R$ is indeed nearly epsilon-strongly graded.
\end{proof}
We also note that the following characterization of graded von Neumann regular rings 
which are nearly epsilon-strongly $S$-graded rings inducing $S$ holds true
(cf.~Theorem~1.1 in \cite{krg}, Proposition~1 in \cite{rh} and 
Proposition~\ref{propositionvn}).
\begin{proposition}\label{propositionvnrg}
Let $R=\bigoplus_{s\in S}R_s$ be an $S$-graded ring inducing $S$ which is nearly
epsilon-strongly graded. Then the following statements are equivalent:
\begin{itemize}
\item[$i)$] $R$ is graded von Neumann regular;
\item[$ii)$] Every principal right (left) homogeneous ideal of $R$ is generated by
a homogeneous idempotent element;
\item[$iii)$] Let $I$ be a right (left) ideal of $R$ which is generated by finitely many
homogeneous elements, say $\{x_1,\dots,x_n\},$ such that for all 
$i\in\{1,\dots,n\}$ we have that $\deg(x_i)(\deg(x_i))^{-1}=e,$ for some 
$e\in I(S)^\times.$ Then $I$ is generated by a homogeneous idempotent element. 
\end{itemize}
\end{proposition}
\begin{proof}
Implications $i)\Rightarrow ii)$ and $iii)\Rightarrow i)$ are clear like in the 
non-graded case. The proof of implication $ii)\Rightarrow iii)$ is also similar to the
non-graded case but there are certain details that should be addressed. So, let 
$x$ and $y$ be nonzero homogeneous elements of $R,$ say $\deg(x)=s,$ 
$\deg(y)=t,$ and such that $ss^{-1}=tt^{-1}=e\in I(S).$ Then, $es=s$ and $et=t.$ 
Since there exist $\epsilon'(x)\in R_{s^{-1}}R_s$ and 
$\epsilon'(y)\in R_{t^{-1}}R_t$ such that $x=x\epsilon'(x)$ and $y=y\epsilon'(y),$ 
we have that $xR+yR$ is a homogeneous right ideal of $R$ generated by the set 
$\{x,y\}.$ Now, by $ii),$ there exists a homogeneous idempotent element $a\in R$ 
such that $xR=aR.$ Hence, $x=ax.$ It follows that $s=\deg(x)=\deg(a)\deg(x).$ 
Since $es=s,$ the cancellativity of $S$ implies that $\deg(a)=e.$ Further, note that 
then $y-ay\in xR+yR$ is a homogeneous element of $R,$ since $et=t.$ Now,
reasoning similarly, there exists a homogeneous idempotent element $b\in R,$ 
orthogonal to $a,$ and of degree $e,$ such that $(y-ay)R=bR.$ It follows that 
$c=b-ba$ is a homogeneous idempotent element of degree $e,$ which is also 
orthogonal to $a,$ and $cR=bR=(y-ay)R.$ Therefore, like in the non-graded case, 
$xR+yR=(a+c)R.$ This concludes the proof, since $\deg(a)=\deg(c)=e$ and $a+c$ is 
an idempotent element of $R.$
\end{proof}
Let us recall, if $R$ is an $S$-graded ring inducing $S,$ then $R$ is said to be 
\emph{graded semiprime} \cite{eig1} if for any homogeneous ideal $I$ of $R$ we 
have that $I^n\subseteq R$ implies that $I\subseteq R,$ where $n$ is a positive 
integer.
\begin{proposition}\label{propositiongs}
Let $R=\bigoplus_{s\in S}R_s$ be an $S$-graded ring inducing $S.$ If $R$ is graded 
von Neumann regular, then the following assertions hold:
\begin{itemize}
\item[$i)$] Every homogeneous right (left) ideal $I$ of $R$ is idempotent, that is, 
$I^2=I;$
\item[$ii)$] Every homogeneous two-sided ideal of $R$ is graded semiprime.
\end{itemize}
If, moreover, $R$ is nearly epsilon-strongly graded, then:
\begin{itemize} 
\item[$iii)$] The graded Jacobson radical $J^g(R)$ of $R$ is zero.
\end{itemize}
\end{proposition}
\begin{proof}
$i)$ and $ii)$ follow as in the non-graded case, see Corollary~1.2 in \cite{krg}. As 
for $iii),$ let us assume that there exists a nonzero homogeneous element 
$0\neq x\in J^g(R)$ of degree, say $s.$ Since $R$ is by the hypothesis nearly 
epsilon-strongly graded, there exists an element $\epsilon'(x)\in R_{s^{-1}}R_s$ 
such that $x=x\epsilon'(x).$ Then $xR$ is a homogeneous right ideal of $R$ 
generated by $x.$ By Proposition~\ref{propositionvnrg}, there exists a 
homogeneous idempotent element $0\neq y\in R$ such that $xR=yR.$ Since $y$ is
a nonzero idempotent element of $R,$ its degree is a nonzero idempotent of $S.$
Let $\deg(y)=e.$ Again, since $R$ is nearly epsilon-strongly graded, and since
$e^{-1}=e,$ by the cancellativity of $S,$ there exists 
$\epsilon'(y)\in R_{e^{-1}}R_e\subseteq R_e$ such that $y=y\epsilon'(y).$ Now, 
$x\in J^g(R),$ and so, we get that $xR=yR\subseteq J^g(R).$ Hence, $y\in J^g(R).$ 
It follows that $y\in J^g(R)\cap R_e=J(R_e),$ according to Theorem~\ref{theoremj}. 
By Lemma~\ref{lemma1}, we have that $R_e$ is von Neumann regular. It follows by 
Corollary~\ref{corollaryj} that $J(R_e)=0.$ Hence, $y=0,$ a contradiction. 
Therefore, indeed $J^g(R)=0,$ as claimed.
\end{proof}
We now state and prove the main result of this section which characterizes graded 
von Neumann regular $S$-graded rings inducing $S$ in terms of nearly 
epsilon-strongly graded rings.
\begin{theorem}\label{theoremm}
Let $R=\bigoplus_{s\in S}R_s$ be an $S$-graded ring inducing $S$ with a 
cancellative $S.$ Then $R$ is graded von Neumann regular if and only if $R$ is
nearly epsilon-strongly graded and each ring component of $R$ is von Neumann
regular.
\end{theorem}
\begin{proof}
$(\Rightarrow)$ Let $R$ be graded von Neumann regular. Then, by 
Proposition~\ref{proposition2} we have that $R$ is nearly epsilon-strongly graded.
Also, Lemma~\ref{lemma1} implies that $R_e$ is von Neumann regular for every
$e\in I(S).$\\
\indent $(\Leftarrow)$ Let us assume that $R$ is nearly epsilon-strongly graded
and that $R_e$ is a von Neumann regular ring for every idempotent element 
$e\in S.$ We may use the approach of the proof of Theorem~3 in \cite{hy} (see 
also the proofs of Lemma~3.9 and Proposition~3.10 in \cite{dl}) in order to prove
that $R$ is graded von Neumann regular. Let $x\in R$ be a nonzero homogeneous 
element of $R,$ say $x\in R_s.$ Since $R$ is nearly epsilon-strongly graded, there
exist unique $s^{-1}\in S$ and unique idempotent elements $e$ and $f$ from $S$
such that $ss^{-1}=e,$ $s^{-1}s=f,$ $fs^{-1}=s^{-1}e=s^{-1},$ and $es=sf=s.$ 
Then, $R_{s^{-1}}x$ is a left ideal of $R_f.$ Indeed, since $s^{-1}s=f,$ we
get that $R_{s^{-1}}x\subseteq R_{s^{-1}}R_s\subseteq R_f.$ It is clear that 
$R_{s^{-1}}x$ is an additive subgroup of $R_f.$ Now, take 
$r_{s^{-1}}x\in R_{s^{-1}}x$ and $r_f\in R_f$ arbitrarily. Then 
$r_fr_{s^{-1}}x=(r_fr_{s^{-1}})x\in(R_fR_{s^{-1}})x\subseteq R_{fs^{-1}}x.$ 
However, $fs^{-1}=s^{-1}.$ Therefore, $r_fr_{s^{-1}}x\in R_{s^{-1}}x.$ So, 
$R_{s^{-1}}x$ is indeed a left ideal of $R_f.$ We claim that $R_{s^{-1}}x$ is 
generated by an idempotent element in $R_f.$ 
By the hypotheses, $R$ is nearly epsilon-strongly graded. Hence, there exists 
$\epsilon'(x)\in R_{s^{-1}}R_s$ such that $x\epsilon'(x)=x.$ It follows that 
$\epsilon'(x)=\sum_{i=1}^nx_iy_i,$ for some $x_i\in R_{s^{-1}}$ and $y_i\in R_s,$ 
where $n$ is some positive integer. We get 
$x=x\epsilon'(x)=\sum_{i=1}^nxx_iy_i=\sum_{i=1}^nr_iy_i,$ where 
$r_i:=xx_i\in R_sR_{s^{-1}}\subseteq R_e,$ $i=1,\dots,n.$ Note that for each $i$
there exists $\epsilon(y_i)\in R_sR_{s^{-1}}\subseteq R_e$ such that 
$y_i=\epsilon(y_i)y_i\in R_ey_i.$ Therefore, for every $t\in S^\times,$ we have 
that $R_t$ is finitely generated as a left $R_g$-module, where $g\in I(S)$ is such 
that $gt=t.$ In particular, $R_{s^{-1}}$ is finitely generated as a left $R_f$-module, 
say by $\{v_1,\dots,v_n\}\subseteq R_{s^{-1}}$ for some positive integer $n.$ 
Then $R_{s^{-1}}x=\sum_{i=1}^nR_fv_ix.$ However, $s^{-1}s=f,$ and so, 
$v_ix\in R_f,$ $i=1,\dots,n.$ Therefore, $R_{s^{-1}}x$ is a finitely generated left 
ideal of $R_f.$ Since $R_f$ is von Neumann regular, by 
Proposition~\ref{propositionvn} we get that $R_{s^{-1}}x$ is indeed generated by 
an idempotent element of $R_f.$ Hence, $R_{s^{-1}}x=R_fa,$ for some 
$a^2=a\in R_f,$ since $a=a^2\in R_fa.$ Therefore, $a=yx,$ for some 
$y\in R_{s^{-1}}.$ Moreover, we have that 
$R_sR_{s^{-1}}x=R_sR_fa\subseteq R_sa$ since $sf=s.$ Since 
$x=\epsilon(x)x\in R_sR_{s^{-1}}x,$ it follows that there exists $z\in R_s$ such 
that $x=za.$ Therefore, $xa=(za)a=za=x,$ and so, $x=xa=xyx.$ Since $x$ was 
chosen arbitrarily, $R$ is indeed graded von Neumann regular.
\end{proof}
We finish this section by characterizing graded von Neumann regular $S$-graded
rings inducing $S,$ which are \emph{epsilon-strongly graded}, thus generalizing 
Corollary~3.11 in \cite{dl} to $S$-graded rings inducing $S.$
\begin{definition}\label{definitionesg}
Let $R=\bigoplus_{s\in S}R_s$ be an $S$-graded ring inducing $S,$ where $S$ is
cancellative. Then $R$ is said to be \emph{epsilon-strongly graded} if the following 
conditions are satisfied:
\begin{itemize}
\item[$i)$] $S$ satisfies $(\mathrm{LRI});$
\item[$ii)$] For every $s\in S,$ there exists $\epsilon(s)\in R_sR_{s^{-1}}$ such
that for every $x\in R_s$ we have that $\epsilon(s)x=x=x\epsilon(s^{-1}).$
\end{itemize}
\end{definition}
\begin{corollary}
Let $R=\bigoplus_{s\in S}R_s$ be an epsilon-strongly $S$-graded ring inducing $S.$
Then $R$ is graded von Neumann regular if and only if each ring component of $R$ is 
von Neumann regular.
\end{corollary}
\begin{proof}
It is obvious that every epsilon-strongly graded ring is nearly epsilon-strongly graded. 
Therefore, the statement is immediate by Theorem~\ref{theoremm}.
\end{proof}
\subsection{Pseudo-unitary $S$-graded rings inducing $S$}
It is already mentioned that a unital strongly group graded ring is graded von
Neumann regular if and only if the ring component is von Neumann regular 
(Corollary~C.I.1.5.3 in \cite{noy2}). The aim of this subsection is to characterize the 
graded von Neumann regularity of strongly graded $S$-graded rings inducing $S,$ 
which are \emph{pseudo-unitary}.
\begin{definition}[\cite{ha}]\label{definition}
Let $R=\bigoplus_{s\in S}R_s$ be an $S$-graded ring inducing $S$ with a 
cancellative $S.$ We say that $R$ is \emph{pseudo-unitary} or \emph{pseudo-unital}
if the following conditions are satisfied:
\begin{itemize}
\item[$i)$] For every $e\in I(S),$ the ring $R_e$ is a ring with unity $1_e;$
\item[$ii)$] For every $x\in H_R$ there exist $e,$ $f\in I(S)$ such that 
$1_ex=x=x1_f.$
\end{itemize}
\end{definition}
\begin{remark}\label{remarku}
Note that if $R$ is a pseudo-unitary $S$-graded ring inducing $S,$ then $ef=0$ for
all distinct idempotent elements $e,$ $f\in S.$ As we have already recalled, the 
Leavitt path algebra of a graph $E$ is unital if and only if $E$ is finite. This is the 
property shared by the pseudo-unitary rings in general, that is, a pseudo-unitary ring 
is a ring with unity $1$ if and only if $I(S)$ is finite, and in that case, 
$1=\sum_{e\in I(S)}1_e,$ see \cite{ha1,ha}.\\ 
\indent Also, note that if $R$ is epsilon-strongly graded, then it is pseudo-unitary. 
This can be verified similarly to the fact that every epsilon-strongly group-graded ring 
is unital (Proposition~3.8 in \cite{dl1}). Moreover, it can be easily seen that every 
strongly graded pseudo-unitary $S$-graded ring inducing $S,$ for which the condition 
(LRI) holds, is epsilon-strongly graded. The converse does not hold as it is known 
from the group graded case (see Example~2.8 in \cite{dl}).
\end{remark}
\indent The notion of pseudo-unitaryness also served as a motivation to introduce 
the notion of a \emph{pseudo-unitary homogeneous semigroup} in \cite{eig8}. For
the study of \emph{homogeneous semigroups} \cite{eig5} in case when 
they are graded by a group or by a small category all of whose morphisms are 
invertible, we refer the reader to \cite{chr,hm}. Cayley graphs of 
homogeneous semigroups are studied in \cite{eig8,eig10}, inspired by the results 
which can be found in \cite{ak11}.\\
\indent A similar concept to pseudo-unitary rings has been used in 
\cite{amr,am,am1}. Namely, let 
$S$ be an l.i.-semigroup, that is, a semigroup with zero $0$ for which there exists a 
set $E$ of nonzero orthogonal idempotent elements such that for every 
$s\in S^\times$ there exist $e,$ $f\in E$ such that $es=sf=s.$ A ring graded by an 
l.i.-semigroup is said to be \emph{locally unital} \cite{amr,am,am1} if $R_0=0$ and if 
for every idempotent element $e\in E$ there exists an element $1_e\in R_e$ such 
that for every $s\in S^\times$ with $esf=s,$ where $e,$ $f\in E,$ and every 
$x\in R_s$ we have that $1_ex=x1_f=x.$ Therefore, 
in case $S$ is a cancellative semigroup, a pseudo-unitary ring is locally unital, and if 
$I(S)$ is moreover finite, it is unital, and vice-versa.
\begin{example}
Let $R$ be a ring with unity $1,$ and let $\mathbb{M}_2(R)$ be the ring of 
$2\times 2$ matrices with the usual matrix addition and multiplication over $R,$ and 
observe $\mathbb{M}_2(R)$ as an $S$-graded ring inducing $S$ as described in 
Example~\ref{examplenesg}. It is strongly graded since $R$ is with unity. The ring 
components are $(\mathbb{M}_2(R))_{(1,1)}$ and $(\mathbb{M}_2(R))_{(2,2)}.$ 
They are both unital with unities 
$1_{(1,1)}=\left(%
\begin{array}{cc}
  1 & 0 \\
  0 & 0 \\
\end{array}%
\right)$
and
$1_{(2,2)}=\left(%
\begin{array}{cc}
  0 & 0 \\
  0 & 1 \\
\end{array}%
\right),$
respectively. Moreover, it can be easily checked that 
$1_{(1,1)}(\mathbb{M}_2(R))_{(1,2)}=(\mathbb{M}_2(R))_{(1,2)}=
(\mathbb{M}_2(R))_{(1,2)}1_{(2,2)}$ and 
$1_{(2,2)}(\mathbb{M}_2(R))_{(2,1)}=(\mathbb{M}_2(R))_{(2,1)}=
(\mathbb{M}_2(R))_{(2,1)}1_{(1,1)}.$ So, $\mathbb{M}_2(R)$ is 
pseudo-unitary since $S$ is cancellative.\\
\indent This can be easily generalized to the full matrix ring $\mathbb{M}_n(R)$
of $n\times n$ matrices over $R,$ for a natural number $n.$ In 
particular, if $K$ is a field, then $\mathbb{M}_n(K)$ and $L_K(A_n),$ where $A_n$ 
is the oriented $n$-line graph having $n$ vertices and $n-1$ edges, are isomorphic 
as $S$-graded rings inducing $S,$ and are graded von Neumann regular (see 
Section~\ref{introduction}). Notice that the ring components of 
$L_K(A_n)$ are isomorphic to $K,$ hence, von Neumann regular, and that for every 
nonzero homogeneous element $x$ there exists a homogeneous element $y$ such 
that $xy$ is a nonzero homogeneous element of an idempotent degree. If $E$ is a 
directed graph, another example of a strongly graded von Neumann regular ring which 
is pseudo-unitary, with these two properties, is the Cohn path algebra of $E,$ that is, 
the contracted semigroup ring $K_0[S(E)]$ (see for instance \cite{hm}) where 
$S(E)$ is the graph inverse semigroup of $E$ (see \cite{ah}). We will soon prove 
that these two conditions are both necessary and sufficient for a strongly graded 
pseudo-unitary $S$-graded ring inducing $S$ to be graded von Neumann regular.
\end{example}
\begin{remark}
It is well-known that a unital ring $R=\bigoplus_{g\in G}R_g,$ graded by a group 
$G,$ is strongly graded if and only if its unity belongs to $R_gR_{g^{-1}}$ for
every $g\in G$ (see for instance Proposition~1.1.1 in \cite{noy}). In the previous
example, $1_{(i,i)}\in(\mathbb{M}_2(R))_{(i,j)}(\mathbb{M}_2(R))_{(j,i)}$ for all 
$i,j=1,2,$ and $S$ satisfies (LRI). Let $R$ be a pseudo-unitary $S$-graded ring 
inducing $S,$ and let us assume that $S$ satisfies $\mathrm{(LRI)}.$ Then we note 
that $R$ is strongly graded if and only if $1_e\in R_sR_{s^{-1}}$ for every 
$e\in I(S)^\times$ and every $s\in S$ such that $ss^{-1}=e.$ Namely, let us assume 
that for every $e\in I(S)^\times$ and every $s\in S$ such that $ss^{-1}=e,$ we 
have that $1_e\in R_sR_{s^{-1}}.$ Let $s,$ $t\in S$ be such that $R_sR_t\neq0.$ 
Also, let $e,$ $f\in I(S)^\times$ be such that $tt^{-1}=e,$ $t^{-1}t=f,$ $et=t=tf,$ 
and $ft^{-1}=t^{-1}=t^{-1}e.$ Since $0\neq R_sR_t=R_s(R_t1_f)\subseteq
R_sR_tR_f,$ it follows by Remark~\ref{remarkassoc} that $st=s(tf)=(st)f.$ Also, 
since $R_sR_t\neq0,$ we have that $R_sR_eR_t\neq0$ since 
$R_sR_t=R_s(1_eR_t)\subseteq R_sR_eR_t.$ Hence, we get that $st=s(et)=(se)t.$ 
The cancellativity of $S$ implies that $se=s.$ Hence, 
$0\neq R_s=R_s1_e\subseteq R_sR_tR_{t^{-1}}.$
Therefore, $s=se=s(tt^{-1})=(st)t^{-1}.$ It follows that 
$R_{st}=R_{st}1_f\subseteq R_{st}R_{t^{-1}}R_t\subseteq R_{(st)t^{-1}}R_t
\subseteq R_sR_t.$ Hence, $R_sR_t=R_{st}.$ The converse statement is obvious.
\end{remark}
\indent We now state and prove a characterization of graded von Neumann regular
$S$-graded rings inducing $S$ which are strongly graded and pseudo-unitary, thus
generalizing a well-known group grading counterpart, namely Corollary~C.I.1.5.3 in 
\cite{noy2}. 
\begin{theorem}\label{corollary}
Let $R=\bigoplus_{s\in S}R_s$ be a strongly graded pseudo-unitary $S$-graded ring 
inducing $S.$ Then $R$ is graded von Neumann regular if and only if the following 
conditions are satisfied:
\begin{itemize}
\item[$a)$] For every nonzero homogeneous element $x\in R$ there exists a 
homogeneous element $y\in R$ such that $xy$ is a nonzero homogeneous element 
of $R$ of an idempotent degree;
\item[$b)$] Each ring component of $R$ is von Neumann regular.
\end{itemize}
\end{theorem}
\begin{proof}
$(\Rightarrow)$ Let $R$ be graded von Neumann regular. Then we get $b)$ by
Theorem~\ref{theoremm}. Moreover, by the same theorem we get that $R$ is nearly 
epsilon-strongly graded. Therefore, $S$ satisfies (LRI). Let $x\neq0$ be a 
homogeneous element of $R,$ say $x\in R_s.$ Also, let $s^{-1}\in S$ and 
$e, f\in I(S)$ be such that $ss^{-1}=e,$ $s^{-1}s=f,$ $es=sf=s$ and 
$fs^{-1}=s^{-1}e=s^{-1}.$ Then, since $R$ is pseudo-unitary,
$x=x1_f.$ By the hypothesis, $R$ is strongly graded. Therefore, 
$1_f\in R_{s^{-1}}R_s.$ It follows that $1_f=\sum_{i=1}^nx_iy_i,$ for some
$x_i\in R_{s^{-1}},$ $y_i\in R_s$ and a positive integer $n.$ Now,
$x=x1_f=\sum_{i=1}^n(xx_i)y_i.$ Since $x\neq0,$ we get that there exists $i$
such that $xx_i\neq0.$ On the other hand, $xx_i\in R_sR_{s^{-1}}=R_e,$ that is,
$xx_i$ is a nonzero homogeneous element of an idempotent degree. Hence, $a)$ 
holds true as well.\\
\indent $(\Leftarrow)$ Let us assume that $a)$ and $b)$ hold. Since $R$ is by the
hypothesis pseudo-unitary, for every $e\in I(S)$ we have that $R_e$ is a ring with 
unity $1_e.$ According to Corollary~\ref{corollaryj}, the condition $b)$
implies that $J(R_e)=0$ for every idempotent element $e\in S.$ This, together
with $a)$ implies that $J^g(R)=0,$ according to Theorem~\ref{theoremj}$b)$. Take 
$s\in S^\times.$ Then $R_s$ is not contained in $J^g(R).$ Hence, by 
Theorem~\ref{theoremj}$c),$ we have that there exist elements $e,$ $f\in I(S)$ 
and an element $s^{-1}\in S$ such that $es=sf=s,$ $fs^{-1}=s^{-1}e=s^{-1},$ 
$ss^{-1}=e,$ and $s^{-1}s=f.$ Therefore, $S$ satisfies (LRI).
It follows that $1_ex=x=x1_f$ for every $x\in R_s.$ However, $R$ is strongly 
graded. Hence, $1_e\in R_e=R_sR_{s^{-1}}$ and $1_f\in R_f=R_{s^{-1}}R_s.$ 
Therefore, $ii)$ of Definition~\ref{definitiones} holds true as well, which implies that 
$R$ is nearly epsilon-strongly graded (as already concluded in 
Remark~\ref{remarku}). Thus, $R$ is graded von Neumann regular by 
Theorem~\ref{theoremm}.
\end{proof}
\section{Proof of Theorem~\ref{theoremmain}}\label{Leavitt}
Throughout this section, we keep the notation and agreements set in 
Section~\ref{gleavitt}. So, $R$ is a unital ring, 
$E=(E^0,E^1,\mathrm{r},\mathrm{s})$ a directed graph, 
$S=M(\mathbb{Z},I)\cup\{0\}$ a Brandt semigroup over $\mathbb{Z}$ (although 
some statements in this section hold true for a Brandt semigroup over an arbitrary 
group), and the Cohn path algebra $C_R^X(E)$ with respect to 
$X\subseteq\mathrm{Reg}(E)$, in particular, the Leavitt path algebra $L_R(E),$ is 
observed as a canonically $S$-graded ring with respect to the weight mapping 
$w:E^*\cup\{\mu^*\ |\ \mu\in E^*\}\to S^\times,$ in the sense of 
Proposition~\ref{propositiongrading} and Definition~\ref{definitionc}. Let us recall 
that $L_R(E)$ is equipped with an anti-graded involution mapping $^*$ that sends 
each $x=\sum_{i=1}^nr_i\mu_i\eta^*_i\in L_R(E)$ to
$x^*=\sum_{i=1}^nr_i\eta_i\mu^*_i\in L_R(E)$ (see the proof of 
Proposition~\ref{propositiongrading}).\\
\indent Following the approach of both \cite{rh} and \cite{dl}, we prove 
Theorem~\ref{theoremmain} by establishing it first for the special case of a finite 
graph. The case of an arbitrary graph is handled by observing $L_R(E)$ as 
a direct limit of Leavitt path algebras over finite graphs (similarly to \cite{rh} and 
\cite{dl}).\\
\indent One of the key steps in the proof is to establish that $L_R(E)$ is nearly 
epsilon-strongly graded.
\subsection{Nearly epsilon-strongly groupoid graded Leavitt path algebras}
\indent Theorem~4.2 in \cite{no} asserts that any group graded Leavitt path 
algebra is nearly epsilon-strongly graded. The approach of preordering 
the set of all monomials $\mu\eta^*,$ where $\mu,$ $\eta\in E^*$ and 
$\mathrm{r}(\mu)=\mathrm{r}(\eta),$ taken in \cite{no}, works perfectly fine in 
proving that the Leavitt path algebra, observed as an $S$-graded ring, is nearly 
epsilon-strongly graded in the sense of Definition~\ref{definitionesg}. We include the
key steps for the readers' convenience.
\begin{definition}[cf.~Definition~4.3 in \cite{no}]
Let $X$ denote the set of all monomials $\mu\eta^*,$ where $\mu,$ 
$\eta\in E^*$ and $\mathrm{r}(\mu)=\mathrm{r}(\eta).$ Take $s\in S$ and put 
$X_s=\{x\in X\ |\ w(x)=s\}.$ Suppose that $\mu,$ $\eta,$ $\zeta,$ 
$\theta\in E^*$ are such that $\mu\eta^*,$ $\zeta\theta^*\in X_s.$ Then we put
$\mu\eta^*\leq\zeta\theta^*$ if and only if $\mu$ is the initial subpath of $\zeta.$ 
Also, we define $\mu\eta^*\sim\zeta\theta^*$ if and only if $\mu=\zeta.$
\end{definition}
Of course, for every $s\in S,$ we have that $(L_R(E))_s$ is the 
$R$-linear span of $X_s.$\\
\indent It can be easily verified that the following statements hold 
(cf.~Proposition~4.2 in \cite{no}):
\begin{itemize}
\item[$a)$] The relation $\leq$ is a preorder on $X_s;$
\item[$b)$] The relation $\sim$ is an equivalence relation on $X_s;$
\item[$c)$] The quotient relation $\preceq$ on the factor set $X_s/\sim,$ induced
by $\leq,$ is a partial order.
\end{itemize}
Let $x=\mu\eta^*\in X_s,$ and let $e\in I(S)$ be such that $ss^{-1}=e.$ According
to the proof of Proposition~\ref{propositiongrading}, we have that $w(xx^*)=e.$
Having this in mind, the proofs of the next two results can be carried out as in 
\cite{no} and are omitted.
\begin{proposition}[cf.~Proposition~4.3 in \cite{no}]
Let $s\in S$ and $e\in I(S)$ be such that $ss^{-1}=e.$ Then, the mapping 
$\mathcal{N}_s: X_s/\sim\to(L_R(E))_e,$
given by $\mathcal{N}_s([x]_\sim)=xx^*$ $([x]_\sim\in X_s/\sim),$ is well
defined.
\end{proposition}
\begin{proposition}[cf.~Proposition~4.4 in \cite{no}]\label{lemmaa}
Let $s\in S$ and $x,$ $y\in X_s.$ Then the following statements hold:
\begin{itemize}
\item[$a)$] If $[x]_\sim\preceq[y]_\sim,$ then $\mathcal{N}_s([x]_\sim)y=y;$
\item[$b)$] If $[x]_\sim\npreceq[y]_\sim$ and $[y]_\sim\npreceq[x]_\sim,$
then $\mathcal{N}_s([x]_\sim)y=0.$
\end{itemize}
\end{proposition}
\begin{theorem}\label{theoremnesg}
Let $R$ be a unital ring, and let $E$ be a directed graph. Then the Leavitt path 
algebra $L_R(E)$ is nearly epsilon-strongly graded as an $S$-graded ring.
\end{theorem}
\begin{proof}
We may follow the technique presented in the proof of Theorem~4.2 in \cite{no}.
Let $s\in S$ and $x\in(L_R(E))_s.$ Then 
$x=\sum_{i=1}^nr_i\mu_i\eta^*_i$ for some
natural number $n,$ and $r_i\in R,$ $\mu_i\eta^*_i\in X_s.$ Since $n$ is finite,
there exists a subset $\{i_1,\dots,i_k\}$ of $\{1,\dots,n\}$ such that
$\{[\mu_{i_j}\eta^*_{i_j}]\}_{j=1}^k$ is the set of minimal elements of the set
$\{[\mu_i\eta^*_i]\}_{i=1}^n$ with respect to $\preceq.$ Let
$\epsilon_s(x):=\sum_{j=1}^k\mathcal{N}_s(\mu_{i_j}\eta^*_{i_j}).$ Then
$\epsilon_s(x)\in(L_R(E))_s(L_R(E))_{s^{-1}}.$ Like in the
proof of Theorem~4.2 in \cite{no}, with the help of Proposition~\ref{lemmaa}, we
get that $\epsilon_s(x)x=x.$ Note that $\epsilon_s(x)^*=\epsilon_s(x).$ 
Since $x^*\in(L_R(E))_{s^{-1}},$ by what we have just established, there exists
$\epsilon_{s^{-1}}(x^*)\in (L_R(E))_{s^{-1}}(L_R(E))_s$ such that 
$\epsilon_{s^{-1}}(x^*)^*=\epsilon_{s^{-1}}(x^*)$ and
$\epsilon_{s^{-1}}(x^*)x^*=x^*.$ Now, since $R$ commutes with the set of
generators of $L_R(E),$ we have that
$x\epsilon_{s^{-1}}(x^*)=(\epsilon_{s^{-1}}(x^*)^*x^*)^*
=(\epsilon_{s^{-1}}(x^*)x^*)^*=(x^*)^*=x.$ Hence, $L_R(E)$ is indeed nearly
epsilon-strongly graded as an $S$-graded ring.
\end{proof}
\begin{remark}
We note that in the case of a finite graph $E,$ the Leavitt path algebra $L_R(E)$ is 
epsilon-strongly graded as an $S$-graded ring (cf.~Theorem~4.1 in \cite{no}).
\end{remark}
\subsection{Case of a finite graph}
In this subsection we prove that Theorem~\ref{theoremmain} holds true for finite
directed graphs. To this end, we describe the ring components of the $S$-graded
Leavitt path algebras of finite graphs and prove they are von Neumann regular if the 
coefficient ring is von Neumann regular. This is done by showing that these 
ring components are ultramatricial, that is, direct limits of matricial rings, which 
generalizes Corollary~2.1.16 in \cite{aasm}, as well as Lemma~4.4 in \cite{dl}, to
canonically $S$-graded Leavitt path algebras. 
The converse holds for an arbitrary non-null graph, that is, if the ring components are 
von Neumann regular, then the coefficient ring is von Neumann regular.\\
\indent Given an idempotent $e\in I(S)^\times,$ we want to describe nonzero 
$(L_R(E))_e.$ We achieve that by describing nonzero $(C_R^X(E))_e,$ where $X$ is 
a subset of $\mathrm{Reg}(E).$ So, let $X\subseteq\mathrm{Reg}(E),$ and let 
$e\in I(S)^\times$ be such that $(C_R^X(E))_e\neq0.$ We start with the following 
lemma.
\begin{lemma}\label{lemmae}
If $\sum_ir_i\mu_i\eta_i^*$ is an element of the $e$-component $(C_R^X(E))_e$ 
of the $S$-graded algebra $C_R^X(E),$ that is, an element of the $R$-linear span of 
$X_e=\{\mu\eta^*\ |\ \mu,\,\eta\in E^*,\, \mathrm{r}(\mu)=\mathrm{r}(\eta),\,
w(\mu\eta^*)=e\},$ then $w(\mu_i)=w(\eta_i)$ for every $i.$
\end{lemma}
\begin{proof}
Let $\mu\eta^*\in X_e$ with $\mu=\alpha_1\dots\alpha_m$ and 
$\eta=\beta_1\dots\beta_n,$ where $\alpha_i,$ $\beta_j\in E^1,$ $i=1,\dots,m,$ 
$j=1,\dots,n.$ Let $\mathrm{s}(\beta_1)=v.$ Then 
\[\mu\eta^*=\alpha_1\dots\alpha_m\beta_n^*\dots\beta_1^*
=\alpha_1\dots\alpha_m\beta_n^*\dots\beta_1^*v\neq0.\] Hence,
\[e=w(\mu\eta^*)=w(\mu\eta^*)w(v)=ew(v).\] Since $S$ is cancellative, it follows
that $w(v)=e.$ By the definition of the weight mapping $w,$ we have that 
$ew(\beta_1)=w(\beta_1).$ Now, 
\[\mu\eta^*\beta_1=\alpha_1\dots\alpha_m\beta_n^*\dots\beta_2^*
\mathrm{r}(\beta_1)\neq0.\]
However, as we know, $\mathrm{r}(\beta_1)=\mathrm{s}(\beta_2).$ Hence, 
\[w(\beta_1)=ew(\beta_1)=w(\alpha_1)\dots 
w(\alpha_m)w(\beta_n)^{-1}\dots w(\beta_2)^{-1}w(\mathrm{s}(\beta_2)).\]
Further,
\[\mu\eta^*\beta_1\beta_2=\alpha_1\dots\alpha_m\beta_n^*\dots\beta_3^*
\mathrm{r}(\beta_2)\neq0.\]
Since $\mathrm{r}(\beta_2)=\mathrm{s}(\beta_3),$ we get
\[w(\beta_1)w(\beta_2)=ew(\beta_1)w(\beta_2)=w(\alpha_1)\dots 
w(\alpha_m)w(\beta_n)^{-1}\dots w(\beta_3)^{-1}w(\mathrm{s}(\beta_3)).\]
Continuing this way, we eventually obtain that
\[w(\beta_1)\dots w(\beta_n)=w(\alpha_1)\dots w(\alpha_m)f,\] where
$f=w(\mathrm{r}(\beta_n))=w(\mathrm{r}(\eta))=w(\mathrm{r}(\mu))
=w(\mathrm{r}(\alpha_m)).$ By the definition of $w,$ we get that
$w(\alpha_m)f=w(\alpha_m).$ Hence,
$w(\beta_1)\dots w(\beta_n)=w(\alpha_1)\dots w(\alpha_m),$ that is,
$w(\mu)=w(\eta).$
\end{proof}
Let $S_E$ be the set of elements $s\in S^\times$ such that $w(\mu)=s$ for some 
path $\mu$ in $E,$ that is, 
\[S_E:=\{s\in S^\times\ |\ (\exists\mu\in E^*)\, w(\mu)=s\},\] and
let \[S_e:=\{s\in S_E\ |\ ss^{-1}=e\}.\] Also, for $s,$ $t\in S_e$ we put $s\leq t$
if and only if every path of weight $t$ contains an initial subpath of weight $s.$
\begin{lemma}\label{lemma2}
$S_e$ is a directed set with respect to $\leq.$
\end{lemma}
\begin{proof}
Obviously, $\leq$ is both reflexive and transitive. Now, let $s$ and $t$ be distinct 
elements of $S_e.$ Then $es=s$ and $et=t.$ Hence, every path of weight $s$ has a 
source of weight $e,$ and every path of weight $t$ has a source of weight $e.$ By 
the very definition of the weight mapping $w,$ it follows that either every path of 
weight $t$ contains an initial subpath of weight $s$ or every path of weight $s$ 
contains an initial subpath of weight $t.$ Equivalently, either $s\leq t$ or $t\leq s.$
Thus, $S_e$ is a directed set with respect to $\leq.$
\end{proof}
Let $e\neq s\in S_e,$ and let $P$ be a finite set of paths of $E$ of weight 
$\leq s.$ For $t\leq s,$ let $P_t$ be the set of initial paths of weight $t$ of 
elements of $P,$ and let $Q_t$ be the set of edges $\alpha$ for which there
exist paths $\mu'$ and $\mu''$ such that $w(\mu'\alpha)=t$ and 
$\mu=\mu'\alpha\mu''\in P.$ If $\mu\in P$ and $w(\mu)\geq t,$ then by 
$\mu_t$ we denote the initial subpath of $\mu$ of weight $t.$ Of course, $P_e$
is the set of the sources of paths from $P.$
\begin{definition}[cf.~Definition~2.1.11 in \cite{aasm}]
We say that $P$ is an \emph{$X$-complete subset} of $E^*$ if the following
conditions are satisfied:
\begin{enumerate}
\item All the paths in $P$ of weight less than $s$ end in a sink;
\item For every $\mu\in P,$ every $t<w(\mu)$ such that 
$\mathrm{r}(\mu_t)\in X,$ and every 
$\alpha\in\mathrm{s}^{-1}(\mathrm{r}(\mu_t)),$ we have that
$\mu_t\alpha=\eta_{tw(\alpha)},$ for some $\eta\in P;$
\item For every $\mu\in P_t\setminus E^0$ $(t<s)$ and every 
$\alpha\in Q_{tw(\alpha)}$ such that $\mathrm{r}(\mu)=\mathrm{s}(\alpha),$
we have that $\mu\alpha\in P_{tw(\alpha)}.$
\end{enumerate}
\end{definition}
\indent If $F$ is a finite $X$-complete subgraph of $E$ and $e\neq s\in S_e,$ then,
by using the way $w$ is defined, it can be shown (cf.~Proposition~2.1.12 in 
\cite{aasm}) that there exists an $X$-complete subset of $E^*$ of paths of weight 
$\leq s$ which contains all the paths of weight $s$ of $F$ as well as all the paths of 
weight $<s$ that end in a sink of $E.$ Then, following the technique presented in the 
proof of Proposition~2.1.14 in \cite{aasm}, one obtains the following result. The 
proof is similar and we omit it.
\begin{proposition}[cf.~Proposition~2.1.14 in \cite{aasm}]\label{propositiond}
Let $E$ be an arbitrary graph, $R$ a unital ring, and let 
$X\subseteq\mathrm{Reg}(E).$ Also, let $P$ be an $X$-complete finite subset of
$E^*$ consisting of paths of weight $\leq s\in S_e.$ Define $G(P)$ to be the 
$R$-linear span of monomials $\mu\eta^*,$ $\mu,$ $\eta\in P,$ such that 
$w(\mu)=w(\eta)\in S_e.$ For $e\neq t\leq s,$ define $F_t(P)$ to be the 
$R$-linear span in $C_R^X(E)$ of the elements 
$\mu\left(v-\sum_{\alpha\in Q_t,\mathrm{s}(\alpha)=v}\alpha\alpha^*\right)
\eta^*,$ where $\mu,$ $\eta\in P_{t'},$ $t'w(\alpha)=t,$ 
$\mathrm{r}(\mu)=\mathrm{r}(\eta)=v\notin X,$ and 
$Q_t\cap\mathrm{s}^{-1}(v)\neq\emptyset.$ Also, define 
\[F(P)=G(P)+\sum_{e\neq t\leq s}F_t(P).\] Then $F(P)$ is a matricial
$R$-algebra. Moreover, $(C_R^X(E))_e$ is the direct limit of the subalgebras
$F(P),$ where $P$ ranges over all the $X$-complete finite subsets of $E^*$ whose
weights belong to $S_e.$
\end{proposition}
\indent Recall that if $A$ is a ring, the full matrix ring over $A$ of $n\times n$ 
matrices is denoted by $\mathbb{M}_n(A).$
\begin{lemma}\label{lemmafdg}
Let $E$ be a finite directed graph, and $R$ a unital von Neumann regular ring. Then, 
for every idempotent element $e\in S,$ the component $(L_R(E))_e$ of
a canonically $S$-graded ring $L_R(E)$ is von Neumann regular.
\end{lemma}
\begin{proof}
Let $e\in S\setminus\{0\}$ be an idempotent element. If 
$(L_R(E))_e=0,$ there is nothing to prove. So, let 
$(L_R(E))_e\neq0.$ Then, by Lemma~\ref{lemmae},
\[(L_R(E))_e=\{\sum_ir_i\mu_i\eta_i^*\ |\ r_i\in R, \mu_i, \eta_i\in E^*, 
\mathrm{r}(\mu_i)=\mathrm{r}(\eta_i), w(\mu_i)=w(\eta_i)\in S_e\}.\] 
According to Lemma~\ref{lemma2}, we have that $S_e$ is a directed set with 
respect to $\leq.$ Now, for every $s\in S_e$ define $D_s$ to be the set
\begin{align}\nonumber
  & \{\sum_ir_i\mu_i\eta_i^*\ |\ r_i\in R, \mu_i, \eta_i\in E^*, 
\mathrm{r}(\mu_i)=\mathrm{r}(\eta_i), w(\mu_i)=w(\eta_i)=s\}\,\cup\\
\nonumber  & \{\sum_ir_i\mu_i\eta_i^*\ |\ r_i\in R, \mu_i, \eta_i\in E^*,
\mathrm{r}(\mu_i)=\mathrm{r}(\eta_i)\in\mathrm{Sink}(E), 
w(\mu_i)=w(\eta_i)<s\}.
\end{align}
It is easy to verify that $D_s$ is an $R$-subalgebra of $(L_R(E))_e.$ On the other 
hand, similarly to Corollary~2.1.16 in \cite{aasm}, one concludes that 
\[D_s\cong\left(\bigoplus_{t<s}\bigoplus_{v\in\mathrm{Sink}(E)}
\mathbb{M}_{|P(t,v)|}(R)\right)\oplus\left(\bigoplus_{v\in E^0}
\mathbb{M}_{|P(s,v)|}(R)\right),\] where $P(s,v)$ denotes the set of all paths 
$\mu\in E^*$ for which $w(\mu)=s$ and $\mathrm{r}(\mu)=v.$
As a corollary to Proposition~\ref{propositiond}, we have that $(L_R(E))_e$
is the direct limit of the subalgebras $D_s,$ where $s$ ranges over $S_e.$ In 
particular, \[(L_R(E))_e=\bigcup_{s\in S_e}D_s.\]  
Therefore, by Proposition~5.2.14 in \cite{bk}, we get that 
$(L_R(E))_e$ is von Neumann regular.
\end{proof}
Let $E$ now be an arbitrary directed graph, distinct from the null graph, and let $R$
be a unital ring. Also, let us assume that $(L_R(E))_e$ is a von Neumann 
regular ring for every idempotent element $e\in S^\times.$ Since $E$ is distinct 
from the null graph, there exists $e\in I(S)^\times$ such that 
$(L_R(E))_e\neq0,$ that is, there exists a vertex $v\in E^0$ such that
$w(v)=e.$ Then $R$ can be embedded in $(L_R(E))_e$ via the mapping 
$x\mapsto xv$ $(x\in R).$ Let $0\neq x\in R.$ Since $(L_R(E))_e$ is von 
Neumann regular, there exists $a\in(L_R(E))_e$ such that 
$xv=(xv)a(xv).$ We have that $a$ is a finite sum $\sum_ir_i\mu_i\eta_i^*$ for 
some $r_i\in R$ and $\mu_i,$ $\eta_i\in E^*$ such that 
$w(\mu_i)=w(\eta_i)=s_i,$ $\mathrm{r}(\mu_i)=\mathrm{r}(\eta_i),$ and
$s_is_i^{-1}=e,$ according to Lemma~\ref{lemmae}. Hence,
$xv=(xv)\left(\sum_ir_i\mu_i\eta_i^*\right)(xv),$ and therefore, 
$xv=\sum_jxr_jx\mu_j\eta_j^*,$ where the sum goes over those $j$ for which
$\mathrm{s}(\mu_j)=\mathrm{s}(\eta_j)=v.$ We may now proceed exactly as in 
the proof of Lemma~4.5 in \cite{dl} in order to conclude that $x=xr_kx$ for some 
$k.$ Therefore, like in the case of the canonical $\mathbb{Z}$-grading (Lemma~4.5 
in \cite{dl}), the converse of the previous lemma holds true as well in the case of the 
canonical $S$-grading, regardless of the cardinality of $E.$
\begin{lemma}\label{lemmavnr}
Let $R$ be a unital ring, $E$ a directed graph which is not null, and let us observe
$L_R(E)$ as a canonically $S$-graded ring. If $(L_R(E))_e$ is von Neumann regular 
for every idempotent element $e\in S^\times,$ then $R$ is von Neumann regular 
too.
\end{lemma}
We are now ready to prove Theorem~\ref{theoremmain} in the case of a finite graph.
\begin{theorem}\label{fdg}
Let $R$ be a unital ring and let $E$ be a finite directed graph which is not null. Then
$L_R(E)$ is graded von Neumann regular as a canonically $S$-graded ring if and only
if $R$ is von Neumann regular.
\end{theorem}
\begin{proof}
$(\Rightarrow)$ According to the `only if' part of Theorem~\ref{theoremm}, we 
have that $\left(L_R(E)\right)_e$ is a von Neumann regular ring for every 
idempotent element $e\in S.$ The ring $R$ is then von Neumann regular by
Lemma~\ref{lemmavnr}.\\
\indent $(\Leftarrow)$ Let $R$ be von Neumann regular. Then 
Lemma~\ref{lemmafdg} implies that $(L_R(E))_e$ is von Neumann 
regular for all $e\in I(S).$ Moreover, according to Theorem~\ref{theoremnesg}, we 
have that $L_R(E)$ is nearly epsilon-strongly graded as an $S$-graded 
ring. Therefore, by the `if part' of Theorem~\ref{theoremm} we get that
$L_R(E)$ is graded von Neumann regular as an $S$-graded ring. 
\end{proof}
\subsection{Case of an arbitrary graph}
Let $E$ be a directed graph and $X\subseteq\mathrm{Reg}(E).$ Then, according to
Definition~1.5.16 in \cite{aasm}, we define a new graph $E(X)$ in the following 
way. Let $Y:=\mathrm{Reg}(E)\setminus X,$ and let us add new vertices 
$Y'=\{v'\ |\ v\in Y\}.$ The set of vertices $(E(X))^0$ is defined to be the disjoint 
union of $E^0$ and $Y',$ while the set of edges $(E(X))^1$ is defined to be the 
disjoint union of $E^1$ and 
$\{\alpha'\ |\ \alpha\in E^1, \mathrm{r}(\alpha)\in Y\},$ where 
$\alpha'$ is a new edge starting from $\mathrm{s}(\alpha)$ and ending in the new 
vertex $\mathrm{r}(\alpha)'\in Y'.$\\
\indent If $R$ is a unital ring, then, according to Proposition~4.8 in \cite{dl} 
(cf.~Theorem~1.5.8 in \cite{aasm}), $C_R^X(E)$ and $L_R(E(X))$ are 
isomorphic as $\mathbb{Z}$-graded rings. This isomorphism 
$\phi: C_R^X(E)\to L_R(E(X))$ is given by 
\[\phi(v)=\begin{cases} v+v' & \quad\mathrm{if}\quad v\in Y,\\
v & \quad\mathrm{if}\quad v\notin Y,
\end{cases}\quad\mathrm{and}\quad
\phi(\alpha)=\begin{cases} \alpha+\alpha' & \quad\mathrm{if}\quad 
r(\alpha)\in Y,\\
\alpha & \quad\mathrm{if}\quad r(\alpha)\notin Y,
\end{cases}\]
where $v\in E^0,$ $\alpha\in E^1.$
Now, let us observe $C_R^X(E)$ as a canonically $S$-graded ring. If $w$ is the 
weight mapping on $E^*\cup\{\mu^*\ |\ \mu\in E^*\}$ of this $S$-grading of 
$C_R^X(E),$ then we put $w(v')=w(v)$ for every $v'\in Y'$ and 
$w(\alpha')=w(\alpha)$ for every 
$\alpha'\in\{\alpha'\ |\ \alpha\in E^1, \mathrm{r}(\alpha)\in Y\}.$ This extends 
$w$ to $E(X)^*\cup\{\mu^*\ |\ \mu\in E(X)^*\}.$ Then $L_R(E(X))$ is a 
canonically $S$-graded ring with respect to $w,$ and the mapping $\phi$ is a 
homogeneous isomorphism with respect to this grading too. Hence, the following 
proposition holds.
\begin{proposition}\label{propositiongr}
Let $R$ be a unital ring and let $E$ be a directed graph. Then $C_R^X(E)$ and
$L_R(E(X))$ are graded isomorphic as $S$-graded rings.
\end{proposition}
According to Lemma~4.9 in \cite{dl}, if $\psi=(\psi^0,\psi^1): (F,Y)\to (E,X)$ is a 
morphism in $\mathcal{G},$ and $R$ a unital ring, then there exists an induced 
$\mathbb{Z}$-graded ring homomorphism $\bar{\psi}: C_R^Y(F)\to C_R^X(E).$
This homomorphism is given by $\bar{\psi}(v)=\psi^0(v),$ 
$\bar{\psi}(\alpha)=\psi^1(\alpha)$ and $\bar{\psi}(\alpha^*)=\psi^1(\alpha)^*$ 
for all $v\in F^0$ and $\alpha\in F^1.$ Now, let $C_R^Y(F)$ be canonically 
$S$-graded, with the weight mapping $w_F.$ Also, let 
$C_R^X(E)$ be canonically $S$-graded with respect to the weight mapping 
$w_E: E^*\cup\{\mu^*\ |\ \mu\in E^*\}\to S$ such that 
$w_E(\psi^0(v))=w_F(v)$ and $w_E(\psi^1(\alpha))=w_F(\alpha)$ for all 
$v\in F^0$ and $\alpha\in F^1.$ Then, clearly, $\bar{\psi}$ is also a morphism in 
the category $S\operatorname{-RING}.$ Therefore, the following lemma holds.
\begin{lemma}
If $\psi: (F,Y)\to (E,X)$ is a morphism in $\mathcal{G},$ and if $R$ is a unital ring, 
then there exists an induced graded ring homomorphism 
$\bar{\psi}: C_R^Y(F)\to C_R^X(E)$ of $S$-graded rings.
\end{lemma}
Let $R$ be a unital ring. Inspired by Definition~4.10 in \cite{dl}, we define the 
$S$-\emph{Cohn path algebra functor}
$C_R: \mathcal{G}\to S\operatorname{-RING}$ by $(E,X)\mapsto C^X_R(E),$ and 
$\psi\mapsto\bar{\psi}$ for all objects $(E,X)$ of $\mathcal{G}$ and all morphisms
$\psi$ of $\mathcal{G}.$ Then, with the help of the previous lemma, one may prove 
that the following result holds by simply observing homogeneous homomorphisms
of $S$-graded rings instead of $\mathbb{Z}$-graded ring homomorphisms in the
proof of Lemma~4.11 in \cite{dl} (see also Proposition~1.6.4 in \cite{aasm}).
\begin{lemma}\label{lemmafdl}
The $S$-Cohn path algebra functor $C_R$ preserves direct limits.
\end{lemma}
\begin{proposition}\label{propositionm}
Let $R$ be a unital ring and let $E$ be a directed graph. If $R$ is von Neumann
regular, then $L_R(E)$ is graded von Neumann regular as an $S$-graded
ring.
\end{proposition}
\begin{proof}
Of course, if $E$ is the null graph, then $L_R(E)$ is $S$-graded 
von Neumann regular. So, let $E$ be distinct from the null graph. According to
Lemma~1.6.9 in \cite{aasm}, there exists a direct system $\{(F_i,Y_i)\ |\ i\in I\},$
where each $F_i$ is a finite graph, direct limit of which is $(E,\mathrm{Reg}(E)).$
By Lemma~\ref{lemmafdl}, we get that $L_R(E)$ and 
$\underrightarrow{\lim}_iC_R^{Y_i}(F_i)$ are graded isomorphic as $S$-graded 
rings. However, by Proposition~\ref{propositiongr}, each $C_R^{Y_i}(F_i)$ is graded 
isomorphic as an $S$-graded ring to $L_R(F_i(Y_i)).$ Therefore, 
$L_R(E)$ and $\underrightarrow{\lim}_iL_R(F_i(Y_i))$ are 
graded isomorphic as $S$-graded rings (cf.~Proposition~4.12 in \cite{dl} and 
Corollary~1.6.11 in \cite{aasm}). On the other hand, since the graphs $F_i$ are 
finite, by Theorem~\ref{fdg}, each $L_R(F_i(Y_i))$ is graded von 
Neumann regular as an $S$-graded ring. Therefore, $L_R(E)$ is indeed 
graded von Neumann regular as an $S$-graded ring by Lemma~\ref{lemmadlvnr}.
\end{proof}
We may now prove Theorem~\ref{theoremmain}.
\begin{proof}[Proof of Theorem~\ref{theoremmain}]
$(\Rightarrow)$ According to the `only if' part of Theorem~\ref{theoremm}, we 
have that $\left(L_R(E)\right)_e$ is a von Neumann regular ring for every 
idempotent element $e\in S.$ Hence, $R$ is von Neumann regular by 
Lemma~\ref{lemmavnr}.\\
\indent $(\Leftarrow)$ If $R$ is a von Neumann regular ring, then 
Proposition~\ref{propositionm} implies that $L_R(E)$ is graded von
Neumann regular as an $S$-graded ring.
\end{proof}
Finally, we list some properties of $S$-graded Leavitt path algebras which are
consequences of Theorem~\ref{theoremmain}, based on the already established
results on groupoid graded von Neumann regular rings.
\begin{theorem}\label{theoremgs}
Let $R$ be a unital von Neumann regular ring, $E$ a directed graph, and observe
$L_R(E)$ as a canonically $S$-graded ring. Then:
\begin{itemize}
\item[$a)$] Every principal right (left) homogeneous ideal of $L_R(E)$ is 
generated by a homogeneous idempotent element;
\item[$b)$] Let $I$ be a right (left) ideal of $L_R(E)$ which is generated by finitely 
many homogeneous elements, say $\{x_1,\dots,x_n\},$ such that for all 
$i\in\{1,\dots,n\}$ we have that $\deg(x_i)(\deg(x_i))^{-1}=e,$ for some 
$e\in I(S)^\times.$ Then $I$ is generated by a homogeneous idempotent element;
\item[$c)$] Every homogeneous right (left) ideal of $L_R(E)$ is 
idempotent;
\item[$d)$] Every two-sided homogeneous ideal of $L_R(E)$ is graded 
semiprime;
\item[$e)$] $J^g(L_R(E))=0.$
\end{itemize}
\end{theorem}
\begin{proof}
Since $R$ is unital and von Neumann regular, Theorem~\ref{theoremmain} implies
that $L_R(E)$ is graded von Neumann regular as an $S$-graded ring.
Therefore, $c)$ and $d)$ follow by $i)$ and $ii)$ of Proposition~\ref{propositiongs},
respectively. Now, the Leavitt path algebra $L_R(E)$ is nearly 
epsilon-strongly graded as an $S$-graded ring according to 
Theorem~\ref{theoremnesg}. Hence, the assertions $a)$ and $b)$ follow by
Proposition~\ref{propositionvnrg}, and $e)$ follows by 
Proposition~\ref{propositiongs}$iii).$
\end{proof}
\bibliographystyle{amsplain}

\flushleft\small{Emil Ili\'{c}-Georgijevi\'{c}\\
University of Sarajevo\\Faculty of Civil Engineering\\
Patriotske lige 30, 71000 Sarajevo, Bosnia and Herzegovina\\
e-mail: emil.ilic.georgijevic@gmail.com\\
ORCID ID: 0000-0002-4667-8322
\end{document}